\documentclass[12pt]{article}

\usepackage{amsfonts}
\usepackage{amsmath}
\usepackage{color}
\usepackage{tikz}

\newtheorem{definition}{Definition}
\newtheorem{theorem}{Theorem}
\newtheorem{proposition}{Proposition}
\newtheorem{lemma}{Lemma}

\newtheorem{cor}{Corollary}

\newenvironment{proof}{\noindent{\it Proof. }\rm}
{\unskip\nobreak\hfil\penalty50\hskip1em\hbox{}
\nobreak\hfill\qed\par\smallskip}
\def\qed{\vrule height1ex width1ex depth0pt}

\begin{document}
\title{Maximum modulus principle for ``holomorphic functions'' on the quantum matrix ball}
%\author{Olga Bershtein}
%\address{}
%\email{}

%\author{Olof Giselsson}
\author{Olga Bershtein\\
{\footnotesize Department of Mathematics}\\
{\footnotesize University of Copenhagen}\\
{\footnotesize Universitetsparken 5, 2100 K{\o}benhavn {\O}}\\
{\footnotesize email: olya.bersht@gmail.com}\\
 Olof Giselsson and Lyudmila Turowska\\
{\footnotesize Department of Mathematical Sciences,}\\
{\footnotesize Chalmers University of Technology and  the University of Gothenburg,}\\
{\footnotesize Gothenburg SE-412 96, Sweden}\\
{\footnotesize email: olofgi@chalmers.se,
email: turowska@chalmers.se}}

\date{}
\maketitle{}
\abstract{We describe the Shilov boundary ideal for a q-analog  of the algebra of holomorphic functions on the unit ball in the space of $n\times n$ matrices and show that its $C^*$-envelope  is isomorphic to the $C^*$-algebra of continuous functions on the quantum unitary group $U_q(n)$. }
\section{Introduction}
In 80s S.L.~Woronowicz introduced the notion of a compact quantum group within the framework of $C^*$-algebras. It was clear from the beginning that one can consider the appearing theory as a counterpart to the theory of compact groups and their representations. For instance, the key Peter-Weyl theorem was proved for the case of compact quantum groups.

The concept of a non-compact quantum group was much less clear, the corresponding theory was in need of the key statements (for instance, it was not clear whether there exists a Haar measure on a noncompact quantum group). At that point L.~Vaksman and his collaborators suggested another approach to representation theory of noncompact quantum groups. The idea was to construct quantum analogs for homogeneous spaces $X$ of noncompact Lie groups $G_0$ and then develop representation theory in connection with these quantum spaces.

The object they started with were Hermitian symmetric spaces of non-compact type. Let us briefly recall some facts about them and representation theory in connection. Under the Harish-Chandra embedding, an irreducible Hermitian symmetric space $X$  of non-compact type allows a realization as a unit ball $\mathbb D$ in a certain normed vector space. The group of biholomorphic automorphisms $G_0=Aut(\mathbb D)$ is a non-compact real Lie group, and $X$ is a homogeneous space of $G_0$. One of the classical approaches to Harish-Chandra modules for non-compact real groups is to derive them from a $G_0$-orbits in certain flag variety $X_c$ (see, e.g. \cite{Wolf}). The flag variety $X_c$ is namely the dual of $X$ (in the sense of symmetric space), which is a Hermitian symmetric space of compact type. An important fact is that there is a unique closed $G_0$-orbit in $X_c$. This orbit namely corresponds to the Shilov boundary $S(\mathbb D)$ of the bounded symmetric domain $\mathbb D$. In particular, if we consider the unit matrix ball $X=\mathbb D_n=\{Z \in \mathrm{Mat}_n| I-ZZ^*>0\}$, its group of symmetries $G_0$ is the group of pseudounitary matrices $SU(n,n)$, the dual Hermitian symmetric space $X_c$ is the Grassmannian $Gr_n(\mathbb C^{2n})$ and the Shilov boundary is the group of unitary $n \times n$-matrices $S(\mathbb D_n)=U_n$. The series of Harish-Chandra modules related to the Shilov boundary are called the principal degenerate series.

Quantum bounded symmetric domains were introduced in \cite{SV2} and studied in the series of papers (see \cite{vaksman-book} and references therein). The authors associated to a bounded symmetric domain $\mathbb D$ in a complex vector space $V$ non-commutative algebras $\mathbb C[V]_q$ and $\mathrm{Pol}(V)_q$  which they treated as the algebras of holomorphic resp. arbitrary polynomials on the quantum vector space $V$; the algebra of continuous functions on the quantum domain $\mathbb D$ are derived then from $\mathrm{Pol}(V)_q$ via some completion procedure.   The main obstacle for developing representation theory for quantum $G_0$ in this setting was that there were no quantum analogs for $G_0$-orbits on flag varieties. The Shilov boundary was a happy exception because it itself is a compact symmetric space. Although the construction of quantum Shilov boundary for an arbitrary quantum bounded symmetric domain is rather nontrivial, in some simpler cases the object was already known. In particular, for the quantum analog of the unit matrix ball $\mathbb D_n$, the corresponding quantum $U_n$, more precisely a $q$-analog $\mathbb C[U_n]_q$ of the algebra of functions on $U_n$, was well studied. In \cite{vaksman_shilov1}, using a purely algebraic approach,  Vaksman defined a $q$-analog of polynomials $\mathrm{Pol}(S(\mathbb D_n))_q$ on the Shilov boundary $U_n=S(\mathbb D_n)$ as a $*$-algebra isomorphic to $\mathbb C[U_n]_q$  and produced a $*$-homomorphism
$$j_q:\mathrm{Pol}({\mathrm{Mat}}_n)_q\to \mathrm{Pol}(S(\mathbb D_n))_q;$$
the latter can be understood as a $q$-analog  of the  operator that restricts the polynomials on $\overline{\mathbb D_n}$  to its Shilov boundary; we refer to the kernel $J_n:=\ker j_q$ as algebraic Shilov boundary ideal. The principal degenerate series of the quantum $SU(n,n)$ related to the Shilov boundary were studied in \cite{Ber}. In the following we want to clarify the connection of the constructed quantum Shilov boundary with its topological (non-commutative) counterpart.
 %it is a $q$-analog of the idealfunctions vanishing on $S(\mathbb D_n)$.
  %admits  a description in terms of generators and relations.

Recall that the classical topological notion of Shilov boundary is used in the study of uniform  algebras $\mathcal A$, that is, closed subalgebras of $C(X)$ of all continuous functions on a compact $X$ that separate points and contain constants; this is the smallest subset $S\subset X$ where the functions in $\mathcal A$ attain their maximum; the latter means that the restriction of the operator $j: C(X)\to C(S)$, $f\mapsto f|_S$, to the subalgebra $\mathcal A$ is an isometry.  In complex analysis the uniform algebras are typically the algebra ${\mathcal A}(X^0)$  of functions holomorphic in the interior $X^0$ of  a compact domain $X$ in a complex vector space; for instance ${\mathcal A}(\mathbb D_n)\subset C(\overline{\mathbb D}_n)$ for the unit matrix ball $\mathbb D_n$.

In the late 1960s W.Arveson initiated in his influential paper \cite{Ar} the study of non-commutative uniform algebras as (non-selfadjoint) subalgebras $\mathcal A$ of $C^*$-algebras $\mathcal B$  and introduced a non-commutative analog of the Shilov boundary; the latter is the largest ideal $J\subset \mathcal B$ such that the quotient map $j:{\mathcal B}\to{\mathcal B}/J$ is a complete isometry when restricted to $\mathcal A$.
A question, which was raised by Vaksman, is  whether the ``algebraically" constructed Shilov boundary for quantum unit matrix ball coincides with the Arveson one. We give an affirmative answer to this question for general value of $n$. The cases $n=1$ and $n=2$ were treated in \cite{vaksman-boundary} and \cite{pro_tur_shilov}, respectively. More precisely, considering a completion $C_F(\overline{\mathbb D}_n)_q$ of $\mathrm{Pol}(\mathrm{Mat}_n)_q$, a $q$-analog of continuous functions in $\overline{\mathbb D}_n$, we prove that the closure of the algebraic Shilov boundary ideal $J_n$ is the Shilov boundary ideal in the sense of Arveson for the subalgebra $A(\mathbb D_n)_q$ which is the closure of the holomorphic polynomials $\mathbb C[\mathrm{Mat}_n]_q$ in the quantum space of $n\times n$ matrices. The pair $(A({\mathbb D}_n)_q, C_F(\overline{\mathbb D}_n)_q)$ is a $q$-analog of the pair $(A(\mathbb D_n), C(\overline{\mathbb D}_n))$.
Note that the quotient $C_F(\overline{\mathbb D}_n)_q/\overline{J}_n$ provides a realization of the $C^*$-envelope of $A({\mathbb D}_n)_q$. We refer the reader to \cite{arveson_notes, dritschel_mccullough, paulsen} for the background and recent development concerning Shilov boundary and $C^*$-envelope.

The paper is organized as follows. After finishing this section by introducing some general notational conventions, in section 2.1-2.3 we introduce and collect some properties of the main objects of our study, the algebras $\mathbb C[\mathrm{Mat}_n]_q$,  $\mathrm{Pol}(\mathrm{Mat}_n)_q$ and $\mathbb C[SU_n]_q$. The $*$-algebra $\mathrm{Pol}(\mathrm{Mat}_n)_q$ possesses a $\mathbb C[SU_n\times SU_n]_q$-comodule structure; it  plays a crucial role in our consideration and is  presented in Lemma \ref{coaction}. In section 2.4 we discuss $*$-representations of $\mathrm{Pol}(\mathrm{Mat}_n)_q$. In particular, we propose a new construction of the Fock $*$-representation $\pi_{F,n}$ of the $*$-algebra; the representation is known to be the only faithful irreducible $*$-representation by bounded operators (see \cite{ssv}). The construction allows to derive a number of consequences about other $*$-representations of $\mathrm{Pol}(\mathrm{Mat}_n)_q$ which are discussed in sections 2 and 3.  The $C^*$-algebra $C_F(\overline{\mathbb D}_n)_q$ is defined to be  the completion of $\mathrm{Pol}(\mathrm{Mat}_n)_q$ with respect to the norm $\|f\|=\|\pi_{F,n}(f)\|$, $f\in\mathrm{Pol}(\mathrm{Mat}_n)_q$, and shown to be a $C^*$-subalgebra of the (spatial) tensor product $C^*(S)^{\otimes n^2}$ of $n^2$ copies of the $C^*$-algebra generated by the unilateral shift $S$ on $\ell^2(\mathbb Z_+)$.
In section 3 we prove the main results of the paper: Theorems \ref{boundary_ideal} and \ref{shilov_boundary}; they state  that $\overline{J}_n$ is the Shilov boundary ideal of $C_F(\overline{\mathbb D}_n)_q$ relative the subalgebra $A({\mathbb D}_n)_q$. As a corollary we obtain that the $C^*$-envelope of $A(\mathbb D_n)_q$ is isomorphic to the $C^*$-algebra of continuous functions on the quantum unitary group $U_q(n)$, i.e. the $C^*$-enveloping algebra of $\mathbb C[U_n]_q$.  Our approach is based on dilation-theoretic arguments. Namely, we construct a $*$-representation of $\mathrm{Pol}(\mathrm{Mat}_n)_q$ which annihilates the ideal $J_n$ and  compresses to the Fock representation when both are  restricted to the holomorphic part $\mathbb C[\mathrm{Mat}_n]_q$. We refer the reader to \cite{dritschel_mccullough} and \cite{arveson_notes} for the dilation ideas to the Shilov boundary ideal and $C^*$-envelopes.

\medskip

We finish this section by recalling standard notation and notions that are used in  the paper.
For a Hilbert space $H$ we let $B(H)$ denote the space of all bounded linear operators on $H$. We shall write $H\otimes K$ for the Hilbertian tensor product of two Hilbert spaces $H$ and $K$ and $H^{\otimes n}$ for the tensor product of $n$-copies of $H$. For an index set $I$, $\{e_i:i\in I\}$ will always stand for  the standard orthonormal basis in the Hilbert space $\ell^2(I)$.
If $A\in B(H)$, $B\in B(K)$, then $A\otimes B$ stands for the operator in $B(H\otimes K)$ given by $A\otimes B(\xi\otimes\eta)=A\xi\otimes B\eta$, $\xi\in H$, $\eta\in K$. If  $W\subset H$ is a closed subspace and $A\in B(H)$ leaves $W$ invariant then  we write $A|_W$ for the restriction of $A$ to $W$.  If ${\mathcal A}$, $\mathcal B$ are $*$-algebras we write, as usual, $\mathcal A\otimes\mathcal B$ for the algebraic tensor product of the algebras; if $\mathcal A$ and $\mathcal B$ are $C^*$-algebras  by ${\mathcal A}\otimes_{\rm min}{\mathcal B}$ we denote their minimal $C^*$-tensor product. Even though we shall always have one of the $C^*$-algebras $\mathcal A$ and $\mathcal B$ nuclear and hence all $C^*$-norms on $\mathcal A\otimes\mathcal B$ will be the same, we shall keep the notation $\mathcal A\otimes_{\rm min} \mathcal B$ in order to distinguish the latter from the algebraic tensor product of $\mathcal A$ and $\mathcal B$. We write $\mathcal A^{\otimes n}$  for the minimal tensor product of $n$ copies of $\mathcal A$.

For a set $V$, we denote, as usual, by $M_n(V)$ the set of all $n\times n$ matrices with
entries in $V$. It is clearly a vector space if $V$ is such. For a map $\phi:V\to W$ between vector spaces $V$ and $W$ we let $\phi^{(n)}((a_{i,j})_{i,j})=
(\phi(a_{i,j}))_{i,j}$ for each $(a_{i,j})_{i,j}\in M_n(V)$.
%Given matrices ${\bf v}_k=(v_{i,j}^k)_{i,j}\in M_n(V_k)$, we define the multiplicative product
% $${\bf v}_1\odot\ldots\odot {\bf v}_m\in M_n(V_1\otimes\ldots\otimes V_m)$$ by "matrix multiplication", i.e.
 %$$ ({\bf v}_1\odot\ldots\odot {\bf v}_m)_{i,j}=\sum_{k_1,\ldots, k_{m-1}=1}^nv_{i,k_1}^1\otimes v_{k_1,k_2}^2\otimes\ldots\otimes v_{k_{m-1},j}^m. %$$

If $\mathcal A$ is a  $*$-algebra, any $*$-homomorphism $\pi:\mathcal A\to B(H)$ is called a bounded $*$-representation. As in this paper we  mostly deal with bounded $*$-representation we shall often omit the word bounded and write simply a $*$-representation.
The set of all bounded $*$-representations of $\mathcal A$ will be denoted by $\text{Rep}(\mathcal A)$.

Let $\pi_i:\mathcal A_i\to B(H_i)$ be a $*$-representation of ${\mathcal A}_i$, $i=1,2$. We  write $\pi_1 \otimes\pi_2:\mathcal A_1\otimes\mathcal A_2\to B(H_1\otimes H_2)$ for the $*$-representation given by
$\pi_1\otimes\pi_2(a_1\otimes a_2)=\pi_1(a_1)\otimes\pi_2(a_2)$, $a_1\in\mathcal A_1$, $a_2\in \mathcal A_2$. It should not be confused with the tensor product $\pi_1\otimes\pi_2$ of $*$-representations of a Hopf $*$-algebra $\mathcal A$ (used in section \ref{suqn}), which is actually the $*$-homomorphism $(\pi_1\otimes\pi_2)\circ\Delta:\mathcal A\to B(H_1\otimes H_1)$, where $\Delta:{\mathcal A}\to\mathcal A\otimes\mathcal A$ is the co-product on $\mathcal A$.

\section{The $*$-algebras $\mathrm{Pol}(\mathrm{Mat}_n)_q$ and $\mathbb C[SU_n]_q$ and their representations}

\subsection{The $*$-algebra $\mathrm{Pol}(\mathrm{Mat}_n)_q$}\label{algebra}

In what follows $\mathbb{C}$ is a ground field and $q \in(0,1)$. We assume
that all the algebras under consideration are unital. Consider the well known algebra
$\mathbb{C}[\mathrm{Mat}_n]_q$ defined by its generators $z_a^\alpha$,
$\alpha,a=1,\dots,n$, and the commutation relations
\begin{flalign}
& z_a^\alpha z_b^\beta-qz_b^\beta z_a^\alpha=0, & a=b \quad \& \quad
\alpha<\beta,& \quad \text{or}\quad a<b \quad \& \quad \alpha=\beta,
\label{zaa1}
\\ & z_a^\alpha z_b^\beta-z_b^\beta z_a^\alpha=0,& \alpha<\beta \quad
\&\quad a>b,& \label{zaa2}
\\ & z_a^\alpha z_b^\beta-z_b^\beta z_a^\alpha-(q-q^{-1})z_a^\beta
z_b^\alpha=0,& \alpha<\beta \quad \& \quad a<b. & \label{zaa3}
\end{flalign}
This algebra is a quantum analogue of the polynomial algebra
$\mathbb{C}[\mathrm{Mat}_n]$ on the space of $n\times n$ matrices. It follows from the Bergman diamond lemma (see \cite{bergman}) that the lexicographically ordered monomials $(z_n^n)^{\gamma_n^n}(z_{n-1}^{n-1})^{\gamma_{n-1}^{n-1}}\ldots(z_1^n)^{\gamma_1^n}\ldots(z_1^1)^{\gamma_1^1}$, $\gamma_a^\alpha\in \mathbb Z_+$, $\alpha,a=1,\ldots, n$, form a basis of the vector space ${\mathbb C}[\mathrm{Mat}_n]_q$. Hence ${\mathbb C}[\mathrm{Mat}_n]_q$ admits a natural grading given by $\text{deg} z_a^\alpha=1$.

In a similar way, introduce the algebra
$\mathbb{C}[\overline{\mathrm{Mat}}_n]_q$, defined by its generators
$(z_a^\alpha)^*$, $\alpha,a=1,\ldots,n$, and the relations
\begin{flalign}
& (z_b^\beta)^*(z_a^\alpha)^* -q(z_a^\alpha)^*(z_b^\beta)^*=0, \quad a=b
\quad \& \quad \alpha<\beta, \qquad \text{or} & a<b \quad \& \quad
\alpha=\beta, \label{zaa1*}
\\ & (z_b^\beta)^*(z_a^\alpha)^*-(z_a^\alpha)^*(z_b^\beta)^*=0,&
\alpha<\beta \quad\&\quad a>b,& \label{zaa2*}
\\ & (z_b^\beta)^*(z_a^\alpha)^*-(z_a^\alpha)^*(z_b^\beta)^*-
(q-q^{-1})(z_b^\alpha)^*(z_a^\beta)^*=0,& \alpha<\beta \quad \& \quad a<b. &
\label{zaa3*}
\end{flalign}
A grading in $\mathbb{C}[\overline{\mathrm{Mat}}_n]_q$ is given by $\text{deg}(z_a^\alpha)^*=-1$.

Finally, consider the algebra $\mathrm{Pol}(\mathrm{Mat}_n)_q$ whose
generators are $z_a^\alpha$, $(z_a^\alpha)^*$, $\alpha,a=1,\dots,n$, and the list of relations is formed by \eqref{zaa1} --
\eqref{zaa3*} and
\begin{flalign}
&(z_b^\beta)^*z_a^\alpha=q^2 \cdot \sum_{a',b'=1}^n
\sum_{\alpha',\beta'=1}^n R_{ba}^{b'a'}R_{\beta \alpha}^{\beta'\alpha'}\cdot
z_{a'}^{\alpha'}(z_{b'}^{\beta'})^*+(1-q^2)\delta_{ab}\delta^{\alpha \beta},
& \label{zaa4}
\end{flalign}
with $\delta_{ab}$, $\delta^{\alpha \beta}$ being the Kronecker symbols, and
$$
R_{ij}^{kl}=
\begin{cases}
q^{-1},& i \ne j \quad \& \quad i=k \quad \& \quad j=l
\\ 1,& i=j=k=l
\\ -(q^{-2}-1), & i=j \quad \& \quad k=l \quad \& \quad l>j
\\ 0,& \text{otherwise}.
\end{cases}
$$
The involution in $\mathrm{Pol}(\mathrm{Mat}_n)_q$ is introduced in the
obvious way: $*:z_a^\alpha \mapsto(z_a^\alpha)^*$.

%An easy inspection of the commutation relations convinces that the map $$z_j^i\mapsto z_{j+n}^{i+n}, \ i,j=1,\ldots, n,$$ defines an embedding of $\mathrm{Pol}(\mathrm{Mat}_n)_q$ into  $\mathrm{Pol}(\mathrm{Mat}_{2n})_q$.
\medskip

 Recall a standard notation for the $q$-determinant of the matrix $\mathbf{z}=(z_a^\alpha)_{\alpha,a=1}^n$:
 \begin{equation*}
\det \nolimits_q\mathbf{z}=\sum_{s \in
S_n}(-q)^{l(s)}z_1^{s(1)}z_2^{s(2)}\ldots z_n^{s(n)},
\end{equation*}
with $l(s)=\mathrm{card}\{(i,j)|\;i<j \quad \&\quad s(i)>s(j) \}$. It is well-known that $\det_q\mathbf{z}$ is in the center of $\mathbb C[\mathrm{Mat}_n]_q$.

In order to introduce the Shilov boundary for the quantum matrix ball we will need the algebra of regular functions on the quantum $GL_n$, denoted by $\mathbb{C}[GL_n]_q$ (see \cite{KlSh}). It is  the localization of $\mathbb{C}[\mathrm{Mat}_n]_q$ with respect to the multiplicative system $(\det_q\mathbf{z})^{\mathbb N}$.
The algebra $\mathbb{C}[GL_n]_q$ possesses a unique involution $*$ given by
\begin{equation}\label{star_pol}
(z_a^\alpha)^*=(-q)^{a+\alpha-2n}(\det\nolimits_q\mathbf{z})^{-1}\det\nolimits_q\mathbf{z}_a^\alpha
\end{equation}
with $\mathbf{z}_a^\alpha$ being the matrix derived from $\mathbf{z}$ by deleting the $\alpha$-th row and $a$-th column. Furthermore,
\begin{equation}\label{detqz}
\det\nolimits_q\mathbf{z}(\det\nolimits_q\mathbf{z})^*=(\det\nolimits_q\mathbf{z})^*\det\nolimits_q\mathbf{z}=q^{-n(n-1)}
\end{equation}
(see \cite[Lemma 2.1]{vaksman_shilov1}).

\begin{theorem}[\cite{vaksman_shilov1}] \label{vaksman_hom} There exists a unique $*$-homomorphism  $$\psi: \mathrm{Pol}(\mathrm{Mat}_n)_q\to (\mathbb{C}[GL_n]_q,\ast)$$ such that $\psi:z_a^\alpha\mapsto z_a^\alpha$, $\alpha,a=1,\ldots,n$.
\end{theorem}
Note that the $*$-algebra $(\mathbb{C}[GL_n]_q,\ast)$ is isomorphic to the $*$-algebra $\mathbb{C}[U_n]_q=(\mathbb{C}[GL_n]_q,\star)$ of regular functions on the quantum $U_n$ (see \cite{koelnik}) with the isomorphism
$\iota:(\mathbb{C}[GL_n]_q,\ast)\to (\mathbb{C}[GL_n]_q,\star)$ given by $\iota:z_a^\alpha\to q^{\alpha-n}z_a^\alpha$, $a,\alpha=1,\ldots, n$; the involution $\star$ in $\mathbb{C}[U_n]_q$ satisfies
$(z_a^\alpha)^\star=(-q)^{a-\alpha}(\det\nolimits_q\mathbf{z})^{-1}\det\nolimits_q\mathbf{z}_a^\alpha$, $a,\alpha=1,\ldots,n$.

\medskip

We  finish the section by a lemma that makes a connection between  $\mathrm{Pol}(\mathrm{Mat}_n)_q$ for different values of $n$.
For $\varphi\in[0,2\pi)$ let
\begin{equation}\label{def_Pi_phi}
\Pi_\varphi( z_j^i)=\begin{cases} q^{-1}z_j^i, \quad i,j <n, \\
e^{i\varphi}, \quad i=j=n,\\
0,\quad  \text{otherwise}.\end{cases}
\end{equation}

\begin{lemma}\label{pivarphi}
\begin{enumerate}
\item The map $\Pi_\varphi$ extends uniquely  to a $*$-homomorphism from $\mathrm{Pol}(\mathrm{Mat}_n)_q$ to
$\mathrm{Pol}(\mathrm{Mat}_{n-1})_q$.
\item The map $z^i_j\mapsto z^{i+n}_{j+n}$, $i,j=1,\ldots,n$  defines an embedding of $\mathrm{Pol}(\mathrm{Mat}_n)_q$ into $\mathrm{Pol}(\mathrm{Mat}_{2n})_q$.
\end{enumerate}
\end{lemma}

\begin{proof}
1. It is enough to check that all relations \eqref{zaa1}-\eqref{zaa4} turn to correct identities under $\Pi_\varphi$. Relations \eqref{zaa1}-\eqref{zaa3} are checked easily. Indeed, if such a relation does not involve elements of the last row or column of $(z^i_j)_{i,j=1}^n$, then it remains the same. If a relation contains $z_n^n$, then either it is a relation of type \eqref{zaa1} and under $\Pi_\varphi$ it transforms to $0=0$, or it is a relation of type \eqref{zaa3} and under $\Pi_\varphi$ it transforms to $q^{-1}e^{i\varphi}z_j^i=q^{-1}e^{i\varphi}z_j^i$. If, finally, a relation involves some element from the last column or row but not $z_n^n$, then it transforms to $0=0$.

Now let us check the relations between holomorphic and antiholomorphic generators, namely, relations of type \eqref{zaa4}.
Let us consider several cases. The first case is $a \neq b$ and $\alpha \neq \beta$. Then the explicitly written commutation relation has the form
$(z_b^\beta)^*z_a^\alpha=z_{a}^{\alpha}(z_{b}^{\beta})^*$, and this relation either survives under $\Pi_{\varphi}$ or turns to the identity $0=0$.

The second case is when one of the pairs of indices $(a,b)$ and $(\alpha, \beta)$ coincides and the other does not. These cases are completely similar, so we will elaborate here only one. Let us suppose that $a=b$. Then the commutation relation can be rewritten in a more explicit way as follows:
$(z_a^\beta)^*z_a^\alpha=q \sum_{a'=1}^n R_{aa}^{a'a'} z_{a'}^{\alpha}(z_{a'}^{\beta})^*$. If $a=n$, then $\Pi_\varphi$ maps this relation to the identity $0=0$, since $\alpha \neq \beta$. Let us now assume that $a<n$. If either $\alpha$ or $\beta$ equals $n$, then also $\Pi_\varphi$ maps this relation to the identity $0=0$. Finally, if all indices are less than $n$, then $\Pi_\varphi$ maps this relation to the relation $q^{-2}(z_a^\beta)^*z_a^\alpha=q \sum_{a'=1}^{n-1} R_{aa}^{a'a'} z_{a'}^{\alpha}q^{-2}(z_{a'}^{\beta})^*$ which holds in $\mathrm{Pol}(\mathrm{Mat}_{n-1})_q$.

Now let us consider the last case where $a=b$ and $\alpha=\beta$. Then the relation can be written more explicitly as
$$(z_a^\alpha)^*z_a^\alpha=q^2 \sum_{a'=1}^n\sum_{\alpha'=1}^n R_{aa}^{a'a'}R_{\alpha\alpha}^{\alpha'\alpha'} z_{a'}^{\alpha'}(z_{a'}^{\alpha'})^*+1-q^2.$$

If $a=\alpha=n$, then we have the standard relation $(z_n^n)^*z_n^n=q^2z_n^n(z_n^n)^*+1-q^2$ which is mapped by $\Pi_\varphi$ to the identity $1=q^2+1-q^2$.

If $a=n$ and $\alpha <n$ (the other case is analogous), we have the relation $$(z_n^\alpha)^*z_n^\alpha=q^2 \sum_{\alpha'=1}^n R_{\alpha\alpha}^{\alpha'\alpha'} z_n^{\alpha'}(z_n^{\alpha'})^*+1-q^2.$$  $\Pi_\varphi$ maps this relation to the identity $0=-q^2(q^{-2}-1)+1-q^2$.

Finally, if $a<n$ and $\alpha<n$, then we can rewrite the relation in the most explicit way as follows:
\begin{multline*}
(z_a^\alpha)^*z_a^\alpha=q^2 z_{a}^{\alpha}(z_{a}^{\alpha})^* - q^2 \sum_{a'=a+1}^n (q^{-2}-1) z_{a'}^{\alpha}(z_{a'}^{\alpha})^* - q^2\sum_{\alpha'=\alpha+1}^n (q^{-2}-1) z_{a}^{\alpha'}(z_{a}^{\alpha'})^*+ \\ q^2 \sum_{a'=a+1}^n\sum_{\alpha'=\alpha+1}^n (q^{-2}-1)^2 z_{a'}^{\alpha'}(z_{a'}^{\alpha'})^*+1-q^2.
\end{multline*}

Applying $\Pi_\varphi$, we get the relation
\begin{multline*}
q^{-2}(z_a^\alpha)^*z_a^\alpha= z_{a}^{\alpha}(z_{a}^{\alpha})^* - \sum_{a'=a+1}^{n-1} (q^{-2}-1) z_{a'}^{\alpha}(z_{a'}^{\alpha})^* - \sum_{\alpha'=\alpha+1}^{n-1} (q^{-2}-1) z_{a}^{\alpha'}(z_{a}^{\alpha'})^*+ \\ \sum_{a'=a+1}^{n-1}\sum_{\alpha'=\alpha+1}^{n-1} (q^{-2}-1)^2 z_{a'}^{\alpha'}(z_{a'}^{\alpha'})^* +q^2(q^{-2}-1)^2 +1-q^2.
\end{multline*}
Obviously, $q^2(q^{-2}-1)^2 +1-q^2=q^{-2}(1-q^2)$, so after multiplying the relation by the common factor $q^2$ we get the corresponding relation in $\mathrm{Pol}(\mathrm{Mat}_{n-1})_q$.
So, we also checked all the relations between holomorphic and antiholomorphic generators of $\mathrm{Pol}(\mathrm{Mat}_n)_q$ and the map $\Pi_{\varphi}$ admits an extension to an algebra morphism. Its uniqueness is obvious.

2. It is a consequence of a similar easy inspection of the commutation relations; we leave the details to the reader.
\end{proof}

\subsection{The quantum universal enveloping algebra $U_q \mathfrak{sl}_N$ and its action on $\mathrm{Pol}(\mathrm{Mat}_n)_q$}

The Drinfeld-Jimbo quantum universal enveloping algebra is among the basic
notions of the quantum group theory. Recall the definition of the Hopf
algebra $U_q \mathfrak{sl}_N$ \cite{Jant}. Let $(a_{i,j})_{i,j=1}^{N-1}$
be the Cartan matrix of $\mathfrak{sl}_N$:
\begin{equation*}
a_{i,j}=
\begin{cases}
2,& i-j=0,
\\ -1,& |i-j|=1,
\\ 0,& \mathrm{otherwise}.
\end{cases}
\end{equation*}
The algebra $U_q \mathfrak{sl}_N$ is determined by the generators $E_i$,
$F_i$, $K_i$, $K_i^{-1}$, $i=1,\ldots,N-1$, and the relations
\begin{gather*}
K_iK_j=K_jK_i,\quad K_iK_i^{-1}=K_i^{-1}K_i=1,\quad
K_iE_j=q^{a_{ij}}E_jK_i,\quad K_iF_j=q^{-a_{ij}}F_jK_i,
\\ E_iF_j-F_jE_i=\delta_{ij}\,(K_i-K_i^{-1})/(q-q^{-1}),
\\ E_i^2E_j-(q+q^{-1})E_iE_jE_i+E_jE_i^2=0,\qquad |i-j|=1,
\\ F_i^2F_j-(q+q^{-1})F_iF_jF_i+F_jF_i^2=0,\qquad |i-j|=1,
\\ E_iE_j-E_jE_i=F_iF_j-F_jF_i=0,\qquad |i-j|\ne 1.
\end{gather*}
The comultiplication $\Delta$, the antipode $S$, and the counit
$\varepsilon$ are determined by
\begin{gather*}
\Delta(E_i)=E_i \otimes 1+K_i \otimes E_i,\quad \Delta(F_i)=F_i \otimes
K_i^{-1}+1 \otimes F_i,\quad \Delta(K_i)=K_i \otimes K_i,\label{comult}
\\ S(E_i)=-K_i^{-1}E_i,\qquad S(F_i)=-F_iK_i,\qquad S(K_i)=K_i^{-1},
\\ \varepsilon(E_i)=\varepsilon(F_i)=0,\qquad \varepsilon(K_i)=1.
\end{gather*}

Let $U_q \mathfrak{su}_{n,N-n}$ denotes
the $*$-Hopf algebra $(U_q \mathfrak{sl}_N,*)$ given by
$$
(K_j^{\pm 1})^*=K_j^{\pm 1},\qquad E_j^*=
\begin{cases}
K_jF_j,& j \ne n,
\\ -K_jF_j,& j=n,
\end{cases}\qquad F_j^*=
\begin{cases}
E_jK_j^{-1},& j \ne n,
\\ -E_jK_j^{-1},& j=n,
\end{cases}
$$
with $j=1,\ldots,N-1$.

An important fact from the theory of quantum bounded symmetric domains is that the $*$-algebra $\mathrm{Pol}(\mathrm{Mat}_n)_q$ possesses a quantum symmetry with respect to $U_q\mathfrak{su}_{n,n}$.   Recall a general notion of a module algebra. Namely, let $A$ be a Hopf algebra, $B$ an $A$-module and an algebra. Then $B$ is called an $A$-module algebra if the multiplication and embedding of the unit are morphisms of $A$-modules. If both $A$ and $B$ possess involutions, then they should satisfy the compatibility condition $(ab)^*=(S(a))^*b^*$ for each $a \in A, b \in B$.
\begin{proposition}[\cite{SV2}]\label{U_qg-module-algebra}
The $*$-algebra $\mathrm{Pol}(\mathrm{Mat}_n)_q$ possesses a $U_q \mathfrak{su}_{n,n}$-module algebra stru\-cture.
\end{proposition}

From here on we will focus on the $*$-subalgebra of $U_q \mathfrak{su}_{n,n}$ generated by all $K_j^{\pm 1}, E_j,F_j$ for $j \neq n$. This subalgebra will be denoted as $U_q \mathfrak{su}_n \otimes U_q \mathfrak{su}_n$. The formulas below completely determine the $U_q \mathfrak{su}_n \otimes U_q \mathfrak{su}_n$-action on $\mathrm{Pol}(\mathrm{Mat}_n)_q$.

For ${a,\alpha=1,\ldots,n}$ and $k < n$ we have
\begin{align}
K_k^{\pm 1}z_a^\alpha&=
\begin{cases}
q^{\pm 1}z_a^\alpha,& a=k,
\\ q^{\mp 1}z_a^\alpha,& a=k+1,
\\ z_a^\alpha,&\mathrm{ otherwise},
\end{cases}\label{K1,1-action}
\\ F_kz_a^\alpha&=q^{1/2}\cdot
\begin{cases}
z_{a+1}^\alpha,& a=k,
\\ 0,&\mathrm{otherwise},
\end{cases}\label{K1,2-action}
\\ E_kz_a^\alpha&=q^{-1/2}\cdot
\begin{cases}
z_{a-1}^\alpha,& a=k+1,
\\ 0,& \mathrm{otherwise},
\end{cases}\label{K1,3-action}
\end{align}
while for $k>n$ we have
\begin{align}
K_k^{\pm 1}z_a^\alpha&=
\begin{cases}
q^{\pm 1}z_a^\alpha,& \alpha=2n-k,
\\ q^{\mp 1}z_a^\alpha,& \alpha=2n-k+1,
\\ z_a^\alpha,&\mathrm{ otherwise},
\end{cases}\label{K2,1-action}
\\ F_kz_a^\alpha&=q^{1/2}\cdot
\begin{cases}
z_a^{\alpha+1},& \alpha=2n-k,
\\ 0,&\mathrm{otherwise},
\end{cases}\label{K2,2-action}
\\ E_kz_a^\alpha&=q^{-1/2}\cdot
\begin{cases}
z_a^{\alpha-1},& \alpha=2n-k+1,
\\ 0,& \mathrm{otherwise}.
\end{cases}\label{K2,3-action}
\end{align}
Recall that the action on other elements of $\mathrm{Pol}(\mathrm{Mat}_n)_q$ can be obtained from the property that $\xi(fg)=\sum_i \xi_i^{(1)}(f)\xi_i^{(2)}(g)$ and $S(\xi)^*(f^*)=(\xi(f))^*$ for $\xi\in U_q \mathfrak{su}_n \otimes U_q \mathfrak{su}_n$, $f$, $g\in \mathrm{Pol}(\mathrm{Mat}_n)_q$ and $\Delta(\xi)=\sum_i\xi_i^{(1)}\otimes\xi_i^{(2)}$ (in the Sweedler notation).

The above formulas show that the action of $U_q \mathfrak{su}_n \otimes U_q \mathfrak{su}_n$ preserves the degree of each element $f \in \mathrm{Pol}(\mathrm{Mat}_n)_q$. As a simple corollary, we have
\begin{lemma}
The $U_q \mathfrak{su}_n \otimes U_q \mathfrak{su}_n$-action in $\mathrm{Pol}(\mathrm{Mat}_n)_q$ is locally finite.
\end{lemma}

\subsection{The $*$-algebra $\mathbb{C}[SU_n]_q$ and its coaction}\label{suqn}

Recall the definition of the Hopf algebra
$\mathbb{C}[SL_n]_q$. It is defined by the generators $\{t_{i,j}:i,j=1,\ldots,n\}$ and the relations
\begin{flalign*}
& t_{\alpha, a}t_{\beta, b}-qt_{\beta, b}t_{\alpha, a}=0, & a=b \quad \& \quad
\alpha<\beta,& \quad \text{or}\quad a<b \quad \& \quad \alpha=\beta,
\\ & t_{\alpha, a}t_{\beta, b}-t_{\beta, b}t_{\alpha, a}=0,& \alpha<\beta \quad
\&\quad a>b,&
\\ & t_{\alpha, a}t_{\beta, b}-t_{\beta, b}t_{\alpha, a}-(q-q^{-1})t_{\beta, a}
t_{\alpha, b}=0,& \alpha<\beta \quad \& \quad a<b, &
\\ & \det \nolimits_q \mathbf{t}=1.
\end{flalign*}
Here $\det_q \mathbf{t}$ is the  q-determinant of the matrix
$\mathbf{t}=(t_{\alpha,a})_{\alpha,a=1}^{n}$.

It is well known (see \cite{KlSh} or other standard book on quantum groups) that
$\mathbb{C}[SU_n]_q\stackrel{\mathrm{def}}{=}(\mathbb{C}[SL_n]_q,\star)$ is
a Hopf $*$-algebra; the comultiplication $\Delta$, the counit $\varepsilon$, the antipode $S$ and the involution $\star$ are defined as follows
\begin{equation*}
\Delta(t_{i,j})=\sum_k t_{i,k}\otimes t_{k,j},\quad \varepsilon(t_{i,j})=\delta_{ij},\quad S(t_{i,j})=(-q)^{i-j}\det\nolimits_q\mathbf{t}_{ji},
\end{equation*}
 and
\begin{equation*}\label{star5}
t_{i,j}^\star=(-q)^{j-i}\det \nolimits_q\mathbf{t}_{ij},
\end{equation*}
where ${\mathbf t}_{ij}$ is the matrix derived from ${\mathbf t}$ by discarding its $i$-th row and $j$-th column. We have, in particular,
\begin{equation}\label{anti-star}
S(t_{i,j})=t_{j,i}^\star \text{ and } S^2(t_{i,j})=q^{2(i-j)}t_{i,j}, \quad i,j=1,\ldots, n
\end{equation}
(\cite[Proposition 9.10]{KlSh}).
From the relations it easily follows that the mapping $t_{i,j}\mapsto q^{i-j}t_{i,j}$ extends to an automorphism $\alpha$ of $\mathbb C[SU_n]_q$ such that $\alpha(t_{i,j}^\star)=q^{j-i}t_{i,j}^\star$.  Combining, for example, this result with (\ref{anti-star}), one can see that
\begin{equation}\label{auto}\theta: t_{i,j}\mapsto q^{j-i}t_{j,i}
\end{equation}
 gives  a $*$-automorphism of $\mathbb C[SU_n]_q$: we have $\theta(t_{i,j})=\alpha(S(t_{i,j})^\star)$.

There is a canonical isomorphism
$\mathbb{C}[SU_n]_q\simeq\mathbb{C}[U_n]_q/\langle\det\nolimits_q\mathbf{z}-1\rangle,$ given by $t_{i,j}\mapsto z^i_j+\langle\det\nolimits_q\mathbf{z}-1\rangle$, $i,j=1,\ldots,n$;  here $\langle\det\nolimits_q\mathbf{z}-1\rangle$ denotes  the two-sided $*$-ideal generated by $\det_q\mathbf{z}-1$ (see \cite{KlSh}).

Theorem \ref{vaksman_hom} and the remark after it give  a $*$-homomorphism $\phi:\mathrm{Pol}(\mathrm{Mat}_n)_q\to\mathbb C[SU_n]_q$ such that
\begin{equation}\label{hom}
\phi(z_j^i)=q^{i-n}t_{i,j},\ i,j=1,\ldots n.
\end{equation}

 Let us recall here the standard construction of $*$-representations of $\mathbb C[SU_n]_q$.
 Consider the Hilbert space $l^2(\mathbb Z_+)$ with the standard basis $\{e_n: n \in \mathbb Z_+\}$. The formulas
\begin{align*}
& \Pi(t_{1,1})e_n=(1-q^{2n})^{1/2}e_{n-1}, \quad && \Pi(t_{1,2})e_n=q^{n+1}e_n,
\\ & \Pi(t_{2,1})e_n=-q^n e_n, \quad && \Pi(t_{2,2})e_n=(1-q^{2n+2})^{1/2}e_{n+1}, \quad n \in \mathbb Z_+
\end{align*}
define an irreducible $*$-representation of $\mathbb C[SU_2]_q$, (see e.g. \cite{KlSh} or \cite{KorS}).

 For $j\in\{1,\ldots,n-1\}$ let $\psi_j: \mathbb C[SU_n]_q \rightarrow \mathbb C[SU_2]_q$ denote the morphism of $*$-Hopf algebras defined on the generators as follows:
\begin{equation*}
\psi_j(t_{a,b})=\begin{cases} t_{a-j+1,b-j+1}, j \leq a,b \leq j+1,
\\ \delta_{ab}, \text{otherwise}.
\end{cases}
\end{equation*}
Recall that the symmetric group $S_n$ is the Weyl group of $\mathfrak{sl}_n$, and denote by $s_1,\ldots,s_{n-1}$ the transpositions $(1,2),\ldots, (n-1,n)$, respectively. For each $s_j$ one can associate an irreducible $*$-representation of $\mathbb C[SU_n]_q$ via $\pi_{s_j}=\Pi \circ \psi_j$.

Let $s$ be  an arbitrary element of the Weyl group, and $s=s_{j_1}\ldots s_{j_k}$ its reduced decomposition. Then one associates to it the $*$-representation of the Hopf $*$-algebra $\mathbb C[SU_n]_q$ by the rule: $$\pi_s =\pi_{s_{j_1}} \otimes \ldots \otimes \pi_{s_{j_k}}.$$ A remarkable result of Soibelman \cite{soibelman1,soibelman2} and Soibelman and Vaksman \cite{vaksman_soibelman} (see also \cite[chapter 3.6]{KorS} for a general result) says that $\pi_s$ is irreducible and  does not depend on the reduced decomposition, i.e. two representations obtained through this construction are unitarily equivalent if and only if they correspond to the same element of the Weyl group; moreover, any irreducible representation of $\mathbb C[SU_n]_q$ differs from some $\pi_s$ by tensor multiplication by a one-dimensional representation.

\medskip

Recall a general fact on the connection between modules and comodules of algebras and their Hopf duals. In the following proposition $A^\circ$ denotes the Hopf dual (or the finite dual) of an algebra $A$ (the definition can be found in \cite{BG}[p.82]).
\begin{proposition}[{\cite{BG}[p.86-87]}]

\begin{itemize}
\item Let A be an algebra and $V$ be a left $A$-module. Then $V$ can be made into a right $A^\circ$-comodule whose associated left $A$-module is $V$ if and only if $V$ is a locally finite $A$-module.
\item $V$ possesses a left $A$-module locally finite algebra structure if and only if there exists a right $A^\circ$-comodule algebra structure on it.
\end{itemize}
\end{proposition}

\begin{lemma} \label{coaction} The map
$$ \mathcal D_n:z_j^i\mapsto\sum_{a,b=1}^n z_b^a\otimes t_{b,j}\otimes t_{a,i}, i,j=1,\ldots,n$$
extends uniquely to a $*$-homomorphism  $$\mathcal D_n:\mathrm{Pol}(\mathrm{Mat}_n)_q\to \mathrm{Pol}(\mathrm{Mat}_n)_q\otimes\mathbb C[SU_n]_q\otimes \mathbb C[SU_n]_q.$$
\end{lemma}
\begin{proof}
Let us focus on the first part $U_q \mathfrak{sl}_n \subset U_q \mathfrak{sl}_n \otimes U_q \mathfrak{sl}_n$, arguments for the second one are the same. Recall that $\mathbb C[SL_n]_q$ is the (finite) dual for the quantum universal enveloping algebra $U_q \mathfrak{sl}_n$. As the $U_q\mathfrak{sl}_n$-action on $\mathrm{Pol}(\mathrm{Mat}_n)_q$ defined by (\ref{K1,1-action})-(\ref{K1,3-action}) is locally finite, we can apply the previous proposition to obtain a  right $\mathbb C[SL_n]_q$-comodule algebra structure on ${\mathrm{Pol}}(\mathrm{Mat}_n)_q$.
We claim that this coaction is given by
\begin{equation}\label{half-coaction}
z_j^i\mapsto \sum_{b=1}^n z_b^i\otimes t_{b,j},\quad 1\leq i,j\leq n.
\end{equation}
To see this it is enough to verify that
\begin{equation}\label{act}
\xi\cdot z^i_j=\sum_{b=1}^n z_b^it_{b,j}(\xi) \text{ for any }\xi\in U_q\mathfrak{sl}_n, 1\leq i,j\leq n.
\end{equation}
For the generators $\xi\in U_q\mathfrak{sl}_n$ this can be easily recovered from the explicit
formulas for the linear functionals $t_{b,j} \in \mathbb C[SL_n]_q \subset (U_q \mathfrak{sl}_n)^*$. One has
\begin{equation*}
t_{i,i+1}(E_i)=q^{-1/2}, \quad t_{i+1,i}(F_i)=q^{1/2}, \quad t_{i,i}(K_i)=q, \quad t_{i,i}(K_{i-1})=q^{-1}, \quad i=1,\ldots,n,
\end{equation*}
and all other evaluations of $t_{b,j}$ on the generators of $U_q \mathfrak{sl}_n$ are zero (see e.g. \cite{KlSh}).

To see that (\ref{act})  holds for any $\xi\in U_q\mathfrak{sl}_n$ it is enough to observe that whenever
$\sum_{b=1}^n z_b^i t_{b,j}(\xi_m)=\xi_m\cdot z_j^i$, $m=1,2$, then
\begin{eqnarray*}
\sum_{b=1}^n z_b^it_{b,j}(\xi_1\xi_2)&=&\sum_{b=1}^nz_b^i\Delta(t_{b,j})(\xi_1\otimes\xi_2)= \sum_{b=1}^nz_b^i\sum_{k=1}^n t_{b,k}(\xi_1)t_{k,j}(\xi_2)\\
&=&\sum_{k=1}^n(\sum_{b=1}^nz_b^it_{b,k}(\xi_1))t_{k,j}(\xi_2)=\sum_{k=1}^n(\xi_1\cdot z_k^i)t_{k,j}(\xi_2)=\xi_1\cdot(\sum_{k=1}^nz_k^it_{k,j}(\xi_2))\\
&=&\xi_1\cdot(\xi_2\cdot z_j^i)=(\xi_1\xi_2)\cdot z_j^i.
\end{eqnarray*}

Since the algebra $\mathrm{Pol}(\mathrm{Mat}_n)_q$ is a $U_q \mathfrak{su}_n$-module algebra, and the involutions in $U_q \mathfrak{su}_n$ and $\mathbb C[SU_n]_q$ are compatible, one deduces that the map \eqref{half-coaction} extends to a $*$-homomorphism.
\end{proof}

\medskip

\medskip

\subsection{On $*$-representation  of $\mathrm{Pol}(\mathrm{Mat}_n)_q$}\label{representations}

\medskip
In this section we shall discuss $*$-representations of the $*$-algebra $\mathrm{Pol}(\mathrm{Mat}_n)_q$.
One of the most important $*$-representations of  $\mathrm{Pol}(\mathrm{Mat}_n)_q$ is the so-called Fock representation which we now define.

Let ${\mathcal H}_{F,n}$ be a $\mathrm{Pol}(\mathrm{Mat}_n)_q$-module with one generator $v_0$ and the defining relations
\begin{equation*}
\pi_{F,n}(z_j^i)^*v_0=0, \quad i,j =1,\ldots,n.
\end{equation*}

It was proved in \cite[Corollary 2.3, Proposition 2.4, Proposition 2.5]{ssv} that there exists a unique sesquilinear form $(\cdot,\cdot)$ on ${\mathcal H}_{F,n}$ such that $(v_0,v_0)=1$ and $(\pi_{F,n}(f)u,v)=(u,\pi_{F,n}(f^*)v)$ for any $f\in\mathrm{Pol}(\mathrm{Mat}_n)_q$, $u,v\in {\mathcal H}_{F,n}$. Moreover, the sesquilinear form is positive definite and the linear mappings $\pi_{F,n}(f)$ define a bounded $*$-representation, called the Fock representation,  of $\mathrm{Pol}(\mathrm{Mat}_n)_q$; it acts  on the Hilbert space $H_{F,n}$ which is the completion  of ${\mathcal H}_{F,n}$ with respect to $(\cdot,\cdot)$.

We say that a $*$-representation $\pi: \mathrm{Pol}(\mathrm{Mat}_n)_q\to B(H) $ has a vacuum vector if there exists  a unit vector $v$ in $H$ such that $\pi(z_j^i)^*v=0$ for any $i,j=1,\ldots,n$; any such vector $v$ is called a {\it vacuum vector} for $\pi$. The Fock representation is the only (up to unitary equivalence) irreducible $*$-representation of $\mathrm{Pol}(\mathrm{Mat}_n)_q$  that possesses a vacuum vector (see \cite{ssv}).

\medskip

Let $C_F(\overline{\mathbb D}_n)_q$ be the $C^*$-subalgebra of $B(H_{F,n})$ generated by the operators $\pi_{F,n}(z_j^i)$, $i,j=1,\ldots,n$.

By \cite{ssv}, $\pi_{F,n}$ is a faithful irreducible representation of the $*$-algebra $\mathrm{Pol}(\mathrm{Mat}_n)_q$. When $n=1$ and $n=2$ it was proved in \cite{sam-prosk,vaksman-boundary} and \cite{pro_tur} that $\|\pi(a)\|\leq\|\pi_{F,n}(a)\|$ for any bounded $*$-representation of
$\mathrm{Pol}(\mathrm{Mat}_n)_q$ and $a\in \mathrm{Pol}(\mathrm{Mat}_n)_q$, giving that $C_F(\overline{\mathbb D}_n)_q$ is isomorphic to the universal enveloping $C^*$-algebra of $\mathrm{Pol}(\mathrm{Mat}_n)_q$.
Recently,  using, in particular, results from the present work, the second author proved in \cite{olof_preprint} that the statement holds for all values of  $n$.

In order to prove the faithfulness of the Fock representation of the $*$-algebra $\mathrm{Pol}(\mathrm{Mat}_n)_q$ the authors of \cite{ssv} gave an explicit construction of the Fock representation: for this they  embed first $\mathrm{Pol}({\mathrm{Mat}}_n)_q$ in a localization of $\mathbb C[SL_{2n}]_q$ with an involution and then use a concrete well understood $*$-representation of the latter $*$-algebra. The embedding is highly non-trivial that makes the construction very difficult to work with.
In this section we propose a new simpler construction of the Fock representation of $\mathrm{Pol}(\mathrm{Mat}_n)_q$ built out of a  representation of ${\mathbb C}[SU_{2n}]_q$.
Using this construction we shall derive a number of consequences concerning the $C^*$-algebra $C_F(\overline{\mathbb D}_n)_q$ and its $*$-representations.

Recall the $*$-representation $\Pi$ of $\mathbb C[SU_2]_q$  from the previous section and
let $C_q,S,d(q):\ell_2(\mathbb Z_+)\to\ell_2(\mathbb Z_+)$ be the operators defined as follows:
\begin{equation}\label{SC}
Se_n=e_{n+1},\quad C_qe_n=\sqrt{1-q^{2n}}e_n,\quad d(q)e_n=q^ne_n.
\end{equation}

We have $$\Pi(t_{1,1})=S^*C_q,\  \Pi(t_{1,2})=qd(q),\  \Pi(t_{2,1})=-d(q),\  \Pi(t_{2,2})=C_qS.$$

As in \cite{ssv} consider the element

\begin{equation}
u=\left(
\begin{array}{cccccccc}
1 & 2 & \dots & n & n+1 & n+2 & \dots & 2n\\
n+1 & n+2 & \dots & 2n & 1 & 2 & \dots & n
\end{array}
\right)
\end{equation}
of the symmetric group $S_{2n}$. This element is the product of $n$ cycles: $u=c_1\cdot c_2\ldots\cdot c_n$, where $c_{k}=s_{k+n-1}s_{k+n-2}\dots s_{k}$ and $s_k$ is the transposition $(k, k+1)$. The concatenation of the expressions for the cycles $c_k$  gives a reduced decomposition of $u=\sigma_1\sigma_2\ldots\sigma_{n^2}$, $\sigma_i\in\{s_1,\ldots,s_{2n-1}\}$.  Let
$$\pi_u=\pi_{\sigma_1}\otimes\pi_{\sigma_2}\otimes\ldots\pi_{\sigma_{n^2}}$$
be the corresponding representation of $\mathbb C[SU_{2n}]_q$ (see section \ref{suqn}).

Consider the map
$$\iota: z_j^i\mapsto q^{i-n}t_{n+i, n+j}\in{\mathbb C}[SU_{2n}]_q, 1\leq i,j\leq n,$$
defined on the generators $\{z_j^i:1\leq i,j\leq n\}$ of $\mathrm{Pol}(\mathrm{Mat}_n)_q$.
It admits a unique extension to a $*$-homomorphism $\iota: \mathrm{Pol}(\mathrm{Mat}_n)_q\to {\mathbb C}[SU_{2n}]_q$ as  the composition  of the embedding of $\mathrm{Pol}({\mathrm Mat}_n)_q$ into $\mathrm{Pol}(\mathrm{Mat}_{2n})_q$ that sends $z_j^i$ to $z_{j+n}^{i+n}$, $1\leq i,j\leq n$ (see Lemma \ref{pivarphi}), and the $*$-homomorphism  $\phi: \mathrm{Pol}(\mathrm{Mat}_{2n})_q\to \mathbb C[SU_{2n}]_q$, $z_j^i\mapsto q^{i-2n}t_{i,j}$ (see (\ref{hom})).

\begin{theorem}\label{Fock_via_suq2n} Let $\mathcal T:=\pi_u\circ\iota:   \mathrm{Pol}(\mathrm{Mat}_n)_q\to B(\ell^2(\mathbb Z_+)^{\otimes n^2})$.
Then $\mathcal T$ is unitary equivalent to the Fock representation $\pi_{F,n}$.
\end{theorem}

In order to understand the action of $\mathcal T$ we shall associate to it a collection of square box diagrams with directed routes in the following way.

For  $1\leq j,k\leq n$ we have
\begin{equation}\label{rep_T}
\mathcal T(z^j_k)=q^{j-n}\sum_{k_{1},\dots,k_{n^{2}-1}=1}^{2n}\pi_{\sigma_{1}}(t_{n+j,k_{1}})\otimes \dots \otimes \pi_{\sigma_{n^2}}(t_{k_{n^{2}-1},n+k}).
\end{equation}
Recall from  section \ref{suqn} that the only non-zero factors in (\ref{rep_T}) are
\begin{eqnarray*}
&&\pi_{s_i}(t_{i,i+1})=qd(q),\ \pi_{s_{i}}(t_{i+1,i})=-d(q),\ \pi_{s_i}(t_{i,i})=S^*C_q,\\
 &&\pi_{s_i}(t_{i+1,i+1})=C_qS,\ \pi_{s_i}(t_{k,k})=I, k\ne i, i+1,
\end{eqnarray*}
where $d(q),C_q,S:\ell^2(\mathbb Z_+)\to \ell^2(\mathbb Z_+)$ are given by (\ref{SC}) and $I$ stands for the identity operator on $\ell^2(\mathbb Z_+)$.

To each non-zero term  $\pi_{\sigma_{1}}(t_{n+j,k_{1}})\otimes \dots \otimes \pi_{\sigma_{n^2}}(t_{k_{n^{2}-1},n+k})$ in (\ref{rep_T}) we associate  a square tableau consisting of $n^2$ equal boxes. Each box will represent a factor in the term. The non-zero factors $\pi_{s_i}(t_{i,i+1})$, $\pi_{s_{i}}(t_{i+1,i})$, $\pi_{s_i}(t_{i,i})$, $\pi_{s_i}(t_{i+1,i+1})$, $\pi_{s_i}(t_{k,k})$, $k\ne i$, $i+1$, will be represented by the following boxes with arrows:

\begin{eqnarray}
\pi_{s_i}(t_{i,i+1})&\leadsto& \begin{tikzpicture}[line width=0.21mm,scale=0.3]
\draw (1,0) -- (0,0) -- (0,1)--(1,1)--(1,0)--(0,0);
\draw [->] (0.1,0.5)--(0.9,0.5);
\end{tikzpicture}\label{arrow_boxes1}\\
\pi_{s_{i}}(t_{i+1,i})&\leadsto&\begin{tikzpicture}[line width=0.21mm,scale=0.3]
\draw (1,0) -- (0,0) -- (0,1)--(1,1)--(1,0)--(0,0);
\draw [->] (1/2,0.1)--(1/2,0.9);
\end{tikzpicture}\label{arrow_boxes2}\\
\pi_{s_i}(t_{i,i})&\leadsto&\begin{tikzpicture}[line width=0.21mm,scale=0.3]
\draw (1,0) -- (0,0) -- (0,1)--(1,1)--(1,0)--(0,0);
\draw [->] (0.1,0.5)--(0.5,0.5)--(0.5,0.9);
\end{tikzpicture}\label{arrow_boxes3}\\
\pi_{s_i}(t_{i+1,i+1})&\leadsto&\begin{tikzpicture}[line width=0.21mm,scale=0.3]
\draw (1,0) -- (0,0) -- (0,1)--(1,1)--(1,0)--(0,0);
\draw [->] (0.5,0.1)--(0.5,0.5)--(0.9,0.5);
\end{tikzpicture}\label{arrow_boxes4}\\
\pi_{s_i}(t_{k,k})&\leadsto&
\begin{tikzpicture}[line width=0.21mm,scale=0.3]
\draw (1,0) -- (0,0) -- (0,1)--(1,1)--(1,0)--(0,0);
\end{tikzpicture}\label{arrow_boxes5}
\end{eqnarray}
 the $i$-th column of the tableau will correspond to the part related to the cycle $c_i=s_{i+n-1}s_{i+n-2}\ldots s_i$ in the decomposition $u=c_1\cdot c_2\cdot\ldots\cdot c_n$, i.e.
 \begin{eqnarray*}
 &&\pi_{\sigma_{n(i-1)+1}}(t_{k_{n(i-1)},k_{n(i-1)+1}})\otimes\ldots\otimes\pi_{\sigma_{ni}}(t_{k_{n(i-1)},k_{n(i-1)+1}})\\
 &&=\pi_{s_{i+n-1}}(t_{k_{n(i-1)},k_{n(i-1)+1}})\otimes\ldots\otimes\pi_{s_i}(t_{k_{n(i-1)},k_{n(i-1)+1}})
 \end{eqnarray*}
where the reading from the left to the right in the tensor product will correspond to the moving up along the column. For instance, if the part of a non-zero term corresponding to the first cycle is given by $$\pi_{s_n}(t_{n+1,n})\otimes \pi_{s_{n-1}}(t_{n,n-1})\otimes\ldots\otimes \pi_{s_{k}}(t_{k+1,k})\otimes\pi_{s_{k-1}}(t_{k,k})\otimes\pi_{s_{k-2}}(t_{k,k})\otimes\ldots\otimes\pi_{s_1}(t_{k,k})$$ we obtain the column:

$$
\begin{tikzpicture}[thick,scale=0.5]
\draw (0,3) -- (0,4)--(1,4)--(1,3);
\draw (0,2) -- (0,3)--(1,3)--(1,2);
\draw (1,1) -- (0,1) -- (0,2)--(1,2)--(1,1)--(0,1);
\draw (1,0) -- (0,0) -- (0,1)--(1,1)--(1,0)--(0,0);
\draw [->] (0.5,1.05)--(0.5,1.5)--(0.95,1.5);
\draw [->] (1/2,0.6)--(1/2,0.95);
\draw (1/2,0.3)--(1/2,0.5);
\draw (1/2,0)--(1/2,0.2);
\draw (0,0) -- (0,-1)--(1,-1)--(1,0);
\draw [->] (1/2,0.05-1)--(1/2,0.95-1);
\draw [->] (1/2,0.05-2)--(1/2,0.95-2);
\draw (0,0-1) -- (0,-1-1)--(1,-1-1)--(1,0-1);
\node at (-1,0.66) {\vdots};
\node at (-2,1.46) {$\pi_{s_{k-1}}(t_{k,k})$};
\node at (-1,2.66) {\vdots};
\node at (-1.9,3.46) {$\pi_{s_{1}}(t_{k,k})$};
\node at (-2.3,-0.46) {$\pi_{s_{n-1}}(t_{n,n-1})$};
\node at (-2.1,-1.46) {$\pi_{s_{n}}(t_{n+1,n})$};
\end{tikzpicture}
$$
It is not difficult to see that following the above algorithm  applied to a non-zero term  of (\ref{rep_T}) we obtain  in the corresponding $n$ by $n$  box tableau  a path, called {\it admissible} and  built out of arrow boxes (\ref{arrow_boxes1})-(\ref{arrow_boxes5}) with the start and end positions $(n,j)$ and $(k,n)$, respectively.
In fact, as  $\pi_{s_l}(t_{m,p})=\delta_{m,p}I$ for $n\geq m>l+1$, we obtain that the elementary tensor
\begin{eqnarray*}
&&\pi_{\sigma_{1}}(t_{n+j,k_{1}})\otimes \dots \otimes \pi_{\sigma_{n(j-1)+1}}(t_{k_{n(j-1)},k_{n(j-1)+1}})=\\
&&\pi_{s_n}(t_{n+j,k_1})\otimes\ldots\otimes \pi_{s_1}(t_{k_{n-1},k_n})\otimes\pi_{s_{n+1}}(t_{k_{n+1},k_{n+2}})\otimes\ldots\otimes\pi_{s_{n+j-1}}(t_{k_{n(j-1)},k_{n(j-1)+1}})
\end{eqnarray*}
is non-zero if and only if
$n+j=k_1=k_2=\ldots =k_{n(j-1)}$ and $k_{n(j-1)+1}$ is either $n+j-1$ or $n+j$ so that  the last factor is either $\pi_{s_{n+j-1}}(t_{n+j, n+j-1})$ or $\pi_{s_{n+j-1}}(t_{n+j,n+j})$ while the others are the identity operators. Hence the first arrow in the diagrams corresponding to a summand in ${\mathcal T}(z_k^j)$ occurs at the position $(n,j)$ and it  is either $\begin{tikzpicture}[line width=0.21mm,scale=0.3]
\draw (1,0) -- (0,0) -- (0,1)--(1,1)--(1,0)--(0,0);
\draw [->] (1/2,0.1)--(1/2,0.9);
\end{tikzpicture}$ or $\begin{tikzpicture}[line width=0.21mm,scale=0.3]
\draw (1,0) -- (0,0) -- (0,1)--(1,1)--(1,0)--(0,0);
\draw [->] (0.5,0.1)--(0.5,0.5)--(0.9,0.5);
\end{tikzpicture}$, respectively.
Similarly, we observe that
\begin{eqnarray*}
&&\pi_{\sigma_{n^2-k+1}}(t_{k_{n^2-k},k_{n^2-k+1}})\otimes \dots \otimes \pi_{\sigma_{n^2}}(t_{k_{n^2-1},n+k})\\
&&=\pi_{s_{n+k-1}}(t_{k_{n^2-k},k_{n^2-k+1}})\otimes\ldots\otimes\pi_{s_{n}}(t_{k_{n^2-1},n+k})
\end{eqnarray*}
is non-zero if and only if $n+k=k_{n^2-1}=\ldots =k_{n^2-k+1}$ and $k_{n^2-k}$ is either $n+k$ or $n+k-1$ so that the first factor is either $\pi_{s_{n+k-1}}(t_{n+k,n+k})$ or $\pi_{s_{n+k-1}}(t_{n+k-1,n+k})$ and the others are the identity operators. Hence the last arrow in the diagrams corresponding to a summand in ${\mathcal T}(z_k^j)$ occurs at the position $(k,n)$  and it is either $\begin{tikzpicture}[line width=0.21mm,scale=0.3]
\draw (1,0) -- (0,0) -- (0,1)--(1,1)--(1,0)--(0,0);
\draw [->] (0.5,0.1)--(0.5,0.5)--(0.9,0.5);
\end{tikzpicture}$ or $\begin{tikzpicture}[line width=0.21mm,scale=0.3]
\draw (1,0) -- (0,0) -- (0,1)--(1,1)--(1,0)--(0,0);
\draw [->] (0.1,0.5)--(0.9,0.5);
\end{tikzpicture}$, respectively.
With a similar analysis it is not difficult to convince oneself that one gets a connected path from $(n,j)$ to $(k,n)$; the details are left to the reader.

We have, in particular,  that for $n=3$ and $j=k=1$ the diagrams representing the non-zero terms in $\mathcal T(z_1^1)$ are given by

$$
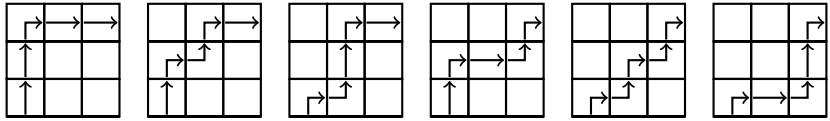
\begin{array}{cccccc}
\begin{tikzpicture}[thick,scale=0.5]
\draw (1,0) -- (0,0) -- (0,1)--(1,1)--(1,0)--(0,0);
\draw (1,1) -- (0,1) -- (0,2)--(1,2)--(1,1)--(0,1);
\draw (1,2) -- (0,2) -- (0,3)--(1,3)--(1,2)--(0,2);
\draw (2,0) -- (1,0) -- (1,1)--(2,1)--(2,0)--(1,0);
\draw (2,1) -- (1,1) -- (1,2)--(2,2)--(2,1)--(1,1);
\draw (2,2) -- (1,2) -- (1,3)--(2,3)--(2,2)--(1,2);
\draw (3,0) -- (2,0) -- (2,1)--(3,1)--(3,0)--(2,0);
\draw (3,1) -- (2,1) -- (2,2)--(3,2)--(3,1)--(2,1);
\draw (3,2) -- (2,2) -- (2,3)--(3,3)--(3,2)--(2,2);
\draw [->] (1/2,0.05)--(1/2,0.95);\draw [->] (1/2,0.05+1)--(1/2,0.95+1);\draw [->] (0.5,0.05+2)--(0.5,0.5+2)--(0.95,0.5+2);\draw [->] (0.05+1,0.5+2)--(0.95+1,0.5+2);\draw [->] (0.05+2,0.5+2)--(0.95+2,0.5+2);
\end{tikzpicture}
&
\begin{tikzpicture}[thick,scale=0.5]
\draw (1,0) -- (0,0) -- (0,1)--(1,1)--(1,0)--(0,0);
\draw (1,1) -- (0,1) -- (0,2)--(1,2)--(1,1)--(0,1);
\draw (1,2) -- (0,2) -- (0,3)--(1,3)--(1,2)--(0,2);
\draw (2,0) -- (1,0) -- (1,1)--(2,1)--(2,0)--(1,0);
\draw (2,1) -- (1,1) -- (1,2)--(2,2)--(2,1)--(1,1);
\draw (2,2) -- (1,2) -- (1,3)--(2,3)--(2,2)--(1,2);
\draw (3,0) -- (2,0) -- (2,1)--(3,1)--(3,0)--(2,0);
\draw (3,1) -- (2,1) -- (2,2)--(3,2)--(3,1)--(2,1);
\draw (3,2) -- (2,2) -- (2,3)--(3,3)--(3,2)--(2,2);
\draw [->] (1/2,0.05)--(1/2,0.95);\draw [->] (0.5,0.05+1)--(0.5,0.5+1)--(0.95,0.5+1);\draw [->] (0.05+1,0.5+1)--(0.5+1,0.5+1)--(0.5+1,0.95+1);\draw [->] (0.5+1,0.05+2)--(0.5+1,0.5+2)--(0.95+1,0.5+2);\draw [->] (0.05+2,0.5+2)--(0.95+2,0.5+2);
\end{tikzpicture}
&
\begin{tikzpicture}[thick,scale=0.5]
\draw (1,0) -- (0,0) -- (0,1)--(1,1)--(1,0)--(0,0);
\draw (1,1) -- (0,1) -- (0,2)--(1,2)--(1,1)--(0,1);
\draw (1,2) -- (0,2) -- (0,3)--(1,3)--(1,2)--(0,2);
\draw (2,0) -- (1,0) -- (1,1)--(2,1)--(2,0)--(1,0);
\draw (2,1) -- (1,1) -- (1,2)--(2,2)--(2,1)--(1,1);
\draw (2,2) -- (1,2) -- (1,3)--(2,3)--(2,2)--(1,2);
\draw (3,0) -- (2,0) -- (2,1)--(3,1)--(3,0)--(2,0);
\draw (3,1) -- (2,1) -- (2,2)--(3,2)--(3,1)--(2,1);
\draw (3,2) -- (2,2) -- (2,3)--(3,3)--(3,2)--(2,2);
\draw [->] (0.5,0.05)--(0.5,0.5)--(0.95,0.5);\draw [->] (0.05+1,0.5)--(0.5+1,0.5)--(0.5+1,0.95);\draw [->] (1/2+1,0.05+1)--(1/2+1,0.95+1);\draw [->] (0.5+1,0.05+2)--(0.5+1,0.5+2)--(0.95+1,0.5+2);\draw [->] (0.05+2,0.5+2)--(0.95+2,0.5+2);
\end{tikzpicture}
&
\begin{tikzpicture}[thick,scale=0.5]
\draw (1,0) -- (0,0) -- (0,1)--(1,1)--(1,0)--(0,0);
\draw (1,1) -- (0,1) -- (0,2)--(1,2)--(1,1)--(0,1);
\draw (1,2) -- (0,2) -- (0,3)--(1,3)--(1,2)--(0,2);
\draw (2,0) -- (1,0) -- (1,1)--(2,1)--(2,0)--(1,0);
\draw (2,1) -- (1,1) -- (1,2)--(2,2)--(2,1)--(1,1);
\draw (2,2) -- (1,2) -- (1,3)--(2,3)--(2,2)--(1,2);
\draw (3,0) -- (2,0) -- (2,1)--(3,1)--(3,0)--(2,0);
\draw (3,1) -- (2,1) -- (2,2)--(3,2)--(3,1)--(2,1);
\draw (3,2) -- (2,2) -- (2,3)--(3,3)--(3,2)--(2,2);
\draw [->] (1/2,0.05)--(1/2,0.95);\draw [->] (0.5,0.05+1)--(0.5,0.5+1)--(0.95,0.5+1);\draw [->] (0.05+2,0.5+1)--(0.5+2,0.5+1)--(0.5+2,0.95+1);\draw [->] (0.5+2,0.05+2)--(0.5+2,0.5+2)--(0.95+2,0.5+2);\draw [->] (0.05+1,0.5+1)--(0.95+1,0.5+1);
\end{tikzpicture}
&
\begin{tikzpicture}[thick,scale=0.5]
\draw (1,0) -- (0,0) -- (0,1)--(1,1)--(1,0)--(0,0);
\draw (1,1) -- (0,1) -- (0,2)--(1,2)--(1,1)--(0,1);
\draw (1,2) -- (0,2) -- (0,3)--(1,3)--(1,2)--(0,2);
\draw (2,0) -- (1,0) -- (1,1)--(2,1)--(2,0)--(1,0);
\draw (2,1) -- (1,1) -- (1,2)--(2,2)--(2,1)--(1,1);
\draw (2,2) -- (1,2) -- (1,3)--(2,3)--(2,2)--(1,2);
\draw (3,0) -- (2,0) -- (2,1)--(3,1)--(3,0)--(2,0);
\draw (3,1) -- (2,1) -- (2,2)--(3,2)--(3,1)--(2,1);
\draw (3,2) -- (2,2) -- (2,3)--(3,3)--(3,2)--(2,2);
\draw [->] (0.5,0.05)--(0.5,0.5)--(0.95,0.5);\draw [->] (0.05+1,0.5)--(0.5+1,0.5)--(0.5+1,0.95);\draw [->] (0.5+1,0.05+1)--(0.5+1,0.5+1)--(0.95+1,0.5+1);\draw [->] (0.05+2,0.5+1)--(0.5+2,0.5+1)--(0.5+2,0.95+1);\draw [->] (0.5+2,0.05+2)--(0.5+2,0.5+2)--(0.95+2,0.5+2);
\end{tikzpicture}
&
\begin{tikzpicture}[thick,scale=0.5]
\draw (1,0) -- (0,0) -- (0,1)--(1,1)--(1,0)--(0,0);
\draw (1,1) -- (0,1) -- (0,2)--(1,2)--(1,1)--(0,1);
\draw (1,2) -- (0,2) -- (0,3)--(1,3)--(1,2)--(0,2);
\draw (2,0) -- (1,0) -- (1,1)--(2,1)--(2,0)--(1,0);
\draw (2,1) -- (1,1) -- (1,2)--(2,2)--(2,1)--(1,1);
\draw (2,2) -- (1,2) -- (1,3)--(2,3)--(2,2)--(1,2);
\draw (3,0) -- (2,0) -- (2,1)--(3,1)--(3,0)--(2,0);
\draw (3,1) -- (2,1) -- (2,2)--(3,2)--(3,1)--(2,1);
\draw (3,2) -- (2,2) -- (2,3)--(3,3)--(3,2)--(2,2);
\draw [->] (1/2+2,0.05+1)--(1/2+2,0.95+1);\draw [->] (0.5,0.05)--(0.5,0.5)--(0.95,0.5);\draw [->] (0.05+2,0.5)--(0.5+2,0.5)--(0.5+2,0.95);\draw [->] (0.5+2,0.05+2)--(0.5+2,0.5+2)--(0.95+2,0.5+2);\draw [->] (0.05+1,0.5)--(0.95+1,0.5);
\end{tikzpicture}\\
\end{array}
$$
The term corresponding,  for example, to the first diagram can be easily recovered as $d(q)\otimes d(q)\otimes C_qS\otimes I\otimes I\otimes d(q)\otimes I\otimes I\otimes  d(q)$.

With the diagram approach the following result becomes evident.

\begin{lemma}\label{vaccum_vector}
The vector $v_0=e_0\otimes\ldots\otimes e_0\in \ell^2(\mathbb Z_+)^{\otimes n^2}$ is a vacuum vector for the $*$-representation $\mathcal T$.
\end{lemma}
 \begin{proof}
 The statement follows from the simple observation that any route corresponding to a non-zero term in (\ref{rep_T}) must contain the right hook arrow box (\ref{arrow_boxes4}) that represents the factor $C_qS$ in the term.  Recalling that $(C_qS)^*e_0=0$, we get the claim.
 \end{proof}

To proceed to the next claim, that can  also be easily derived from the associated square box diagrams, we enumerate the boxes in $n$ by $n$ tableaux by numbers from $1$ to $n^2$  counting the boxes  in each column from the bottom up and moving gradually from the first column to the last one, i.e. we have the following ordering for the box positions in the tableaux:
 \begin{eqnarray*}
 (n,1)&<&(n-1,1)<\ldots<(1,1)<(n,2)<(n-1,2)\ldots<(1,2)<\ldots\\
 &<&(n,n)<(n-1,n)<\ldots<(1,n)
 \end{eqnarray*}
 Let $h(k,l)$ be the number in $\{1,\ldots,n^2\}$ corresponding to the box in the position $(k,l)$. We have $h(k,l)=n(l-1)+n-k+1$ although we will not use this expression.

 \begin{lemma}\label{density}
 The subspace $\{\mathcal T(a)v_0: a\in \mathrm{Pol}(\mathrm{Mat}_n)_q\}$ is dense in $\ell^2(\mathbb Z_+)^{\otimes n^2}$.
 \end{lemma}
 \begin{proof}
Consider all admissible routes from $(n,j)$ to $(k,n)$ in our set of  tableaux. We observe that the only route which does not contain the upper hook arrow box (\ref{arrow_boxes3}), corresponding to the factor $S^*C_q$, is of the form
$$
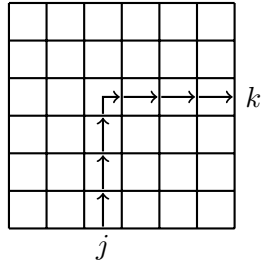
\begin{tikzpicture}[thick,scale=0.5]
\draw[step=1.0,black,thick] (1,1) grid (7,7);
\node at (3.5,0.5) {$j$};
\node at (7.5,4.5) {$k$};
\draw [->] (1/2+3,0.05+1)--(1/2+3,0.95+1);
\draw [->] (1/2+3,0.05+2)--(1/2+3,0.95+2);
\draw [->] (1/2+3,0.05+3)--(1/2+3,0.95+3);
\draw [->] (0.5+3,0.05+4)--(0.5+3,0.5+4)--(0.95+3,0.5+4);
\draw [->] (0.05+4,0.5+4)--(0.95+4,0.5+4);
\draw [->] (0.05+5,0.5+4)--(0.95+5,0.5+4);
\draw [->] (0.05+6,0.5+4)--(0.95+6,0.5+4);
\end{tikzpicture}
$$
The route represents an operator of the form $R_k^j\otimes \underset{\stackrel{\uparrow}{h(k,j)}}{C_qS}\otimes F_k^j$, where $R_k^j$, $F_k^j$ are tensor products of the factors $d(q)$ and $I$.
Furthermore, $S^*C_q$ can occur as a factor in a non-zero term for ${\mathcal T}(z_k^j)$ at the position with  index which is strictly larger than $h(k,j)$.
Recalling now that $S^*C_q$ annihilates $e_0$ we obtain that for $v=e_{m_1}\otimes\ldots \otimes e_{m_{h(k,j)}}\otimes e_0\otimes \ldots\otimes e_0$
$$\mathcal T(z^j_k)v=q^{j-n}(R_k^j\otimes \underset{\stackrel{\uparrow}{h(k,j)}}{C_qS}\otimes F_k^j)v,$$
i.e.  the only summand in $\mathcal T(z^j_k)$ that corresponds to the above route survives after applying $\mathcal T(z^j_k)$ to the vector $v$.
Similarly, for $m\geq 1$,
\begin{eqnarray*}\mathcal T(z^j_k)^mv&=&q^{(j-n)m}((R_k^j)^m\otimes \underset{\stackrel{\uparrow}{h(k,j)}}{(C_qS)^m}\otimes (F_k^j)^m)v\\
&=&\beta_k^j(m)e_{m_1}\otimes\ldots\otimes e_{m_{h(k,j)-1}}\otimes \underset{\stackrel{\uparrow}{h(k,j)}}{e_{m_{h(k,j)}+m}}\otimes e_0\otimes\ldots\otimes e_0
\end{eqnarray*}
for some non-zero constant $\beta_k^j(m)$ (depending on $m_1,\ldots, m_{h(k,j)}$).

By letting $X_{h(k,j)}=\mathcal T(z^j_k)$ we can now easily obtain that
\begin{eqnarray*}
&&X_{n^2}^{m_{n^2}}\ldots X_2^{m_2}X_1^{m_1}v_0=\beta_1(m_1)X_{n^2}^{m_{n^2}}\ldots X_2^{m_2}(e_{m_1}\otimes e_0\otimes e_0\otimes\ldots\otimes e_0)\\
&=&\beta_1(m_1)\beta_2(m_1,m_2)X_{n^2}^{m_{n^2}}\ldots X_3^{m_3}(e_{m_1}\otimes e_{m_2}\otimes e_0\otimes\ldots\otimes e_0)=\ldots \\
&=&\prod_{i=1}^{n^2}\beta_i(m_1,\ldots,m_i)e_{m_1}\otimes e_{m_2}\otimes e_{m_3}\otimes\ldots\otimes e_{m_{n^2}}
\end{eqnarray*}
for non-zero $\beta_i(m_1,\ldots, m_i)$, $i=1,\ldots, n^2$.
The proof is done.
 \end{proof}

The proof above gives more, namely if $X_i$, $i=1,\ldots, n^2$, are as in the proof then we obtain
\begin{cor}
$\{X_{n^2}^{m_{n^2}}\ldots X_2^{m_2}X_1^{m_1}v_0, m_i\in\mathbb Z_+\}$ is an orthogonal basis of $\ell^2(\mathbb Z_+)^{n^2}$.
\end{cor}

 We can now complete the proof of the theorem.

 \medskip

 {\bf Proof of Theorem \ref{Fock_via_suq2n}}. As the Fock representation is the only $*$-representation (up to unitary equivalence) that has a cyclic vacuum vector (see \cite{ssv}), the statement follows from Lemma \ref{vaccum_vector} and Lemma \ref{density}.\qed

 \medskip

 We proceed with deriving a number of useful corollaries from our construction of the Fock representation.

Let $C^*(S)$ be the $C^*$-algebra generated by the isometry $S$.
It is easy to see that the operators $C_q$, $S$, $d(q)\in B(\ell^2(\mathbb Z_+))$ satisfy the following relations
\begin{eqnarray}
C_q&=&(1-q^2)^{1/2}(\sum_{n=0}^\infty q^{2n}S^{n+1}(S^*)^{n+1})^{1/2},\label{cdq1}\\
d(q)&=&\sum_{n=0}^\infty q^n(S^n(S^*)^n-S^{n+1}(S^*)^{n+1}),\quad S^0:=1,\label{cdq2}
\end{eqnarray}
where the series  converge in the operator norm topology.
Consequently,  $C_q,d(q)\in C^*(S)$.

 Identifying $\pi_{F,n}$ with the representation $\mathcal T$ of Theorem \ref{Fock_via_suq2n}, we have
 \begin{cor}\label{C*S}
 $C_F(\overline{\mathbb D}_n)_q\subset C^*(S)^{\otimes{n^2}}.$
 %\otimes_{\rm min}\otimes \ldots\otimes_{\rm min} C^*(S).$
 \end{cor}
 \begin{proof}
 Each $\mathcal T(z_k^j)$ is  a linear combination of elementary tensors  in $B(\ell^2(\mathbb Z_+))^{\otimes n^2}$ each factor of which is either $d(q)$, $C_qS$, $S^*C_q$ or $I$, the latter are elements of $C^*(S)$.
 \end{proof}

Let $\mathcal{K}$ be the space of compact operators on $\ell^2(\mathbb Z_+)$ and $C(\mathbb T)$ be the $C^*$-algebra of continuous functions on the torus $\mathbb T=\{z\in\mathbb C:|z|=1\}$.  It is known that
\begin{equation}\label{def_phi}
\phi: C^*(S)\to C(\mathbb T), S\mapsto z,
\end{equation}
 is a surjective $*$-homomorphism with kernel $\ker\phi=\mathcal K$ (see e.g. \cite{davidson}). In particular, by (\ref{cdq1})-(\ref{cdq2})
\begin{equation}\label{ck}
\phi(d(q))=0\text{ and } \phi(C_q)=1.
\end{equation}

 This is essential for the proof of the following lemma.

\begin{lemma}\label{character}
Let ${\varphi}=(\varphi_1,\ldots,\varphi_n)\in[0,2\pi)^n$ and $\chi_{\bf\varphi}:z_k^l\to e^{i\varphi_k}q^{l-n}\delta_{kl}$, $k,l=1\ldots,n$. Then $\chi_{\bf \varphi}$ extends uniquely to a $*$-representation of $\mathrm{Pol}(\mathrm{Mat}_n)_q$. Moreover,
$|\chi_{\bf \varphi}(a)|\leq\|\pi_{F,n}(a)\|$ for any $a\in \mathrm{Pol}(\mathrm{Mat}_n)_q$, i.e. the map $\pi_{F,n}(a)\mapsto\chi_{\bf \varphi}(a)$, $a\in\mathrm{Pol}(\mathrm{Mat}_n)_q$, extends to a $*$-representation of $C_{F}(\overline{\mathbb D}_n)_q$.
\end{lemma}
\begin{proof}
If we compose the $*$-homomorphism $\phi$ given by \eqref{def_phi} with the evaluation map at $e^{i\alpha}$, $\alpha\in [0,2\pi),$ we get a $*$-homomorphism $C^{*}(S)\rightarrow \mathbb{C}$ determined by $S\mapsto e^{i\alpha}.$ For $\tau=(\tau_1,\ldots,\tau_{n^2})$ let  $\Delta_{\tau}:C^{*}(S)^{\otimes n^{2}}\rightarrow \mathbb{C}$ be a $*$-homomorphism defined by $$I\otimes ...\otimes I\otimes \underset{\stackrel{\uparrow}{\text{i'th place}}}S\otimes I\otimes ...\otimes I\mapsto e^{i\tau_{i}},$$
 $i\in\{1,2,\dots,n^{2}\}.$   As $\phi(d(q))=0$ we obtain that $\Delta_{\tau}$ annihilates any elementary tensor containing $d(q)$ as a factor.
 Next we observe  that the only admissible paths from $(n,l)$ to $(k,n)$ that do not contain the arrow boxes $\begin{tikzpicture}[line width=0.21mm,scale=0.3]
\draw (1,0) -- (0,0) -- (0,1)--(1,1)--(1,0)--(0,0);
\draw [->] (0.1,0.5)--(0.9,0.5);
\end{tikzpicture}$ and $\begin{tikzpicture}[line width=0.21mm,scale=0.3]
\draw (1,0) -- (0,0) -- (0,1)--(1,1)--(1,0)--(0,0);
\draw [->] (1/2,0.1)--(1/2,0.9);
\end{tikzpicture}$ are the paths from $(n,k)$ to $(k,n)$ consisting of  $\begin{tikzpicture}[line width=0.21mm,scale=0.3]
\draw (1,0) -- (0,0) -- (0,1)--(1,1)--(1,0)--(0,0);
\draw [->] (0.5,0.1)--(0.5,0.5)--(0.9,0.5);
\end{tikzpicture}$ and $\begin{tikzpicture}[line width=0.21mm,scale=0.3]
\draw (1,0) -- (0,0) -- (0,1)--(1,1)--(1,0)--(0,0);
\draw [->] (0.1,0.5)--(0.5,0.5)--(0.5,0.9);
\end{tikzpicture}$ following one after another; for example if $n=3$, $k=1$ we have
$$
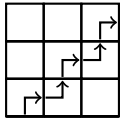
\begin{tikzpicture}[thick,scale=0.5]
\draw (1,0) -- (0,0) -- (0,1)--(1,1)--(1,0)--(0,0);
\draw (1,1) -- (0,1) -- (0,2)--(1,2)--(1,1)--(0,1);
\draw (1,2) -- (0,2) -- (0,3)--(1,3)--(1,2)--(0,2);
\draw (2,0) -- (1,0) -- (1,1)--(2,1)--(2,0)--(1,0);
\draw (2,1) -- (1,1) -- (1,2)--(2,2)--(2,1)--(1,1);
\draw (2,2) -- (1,2) -- (1,3)--(2,3)--(2,2)--(1,2);
\draw (3,0) -- (2,0) -- (2,1)--(3,1)--(3,0)--(2,0);
\draw (3,1) -- (2,1) -- (2,2)--(3,2)--(3,1)--(2,1);
\draw (3,2) -- (2,2) -- (2,3)--(3,3)--(3,2)--(2,2);
\draw [->] (0.5,0.05)--(0.5,0.5)--(0.95,0.5);\draw [->] (0.05+1,0.5)--(0.5+1,0.5)--(0.5+1,0.95);\draw [->] (0.5+1,0.05+1)--(0.5+1,0.5+1)--(0.95+1,0.5+1);\draw [->] (0.05+2,0.5+1)--(0.5+2,0.5+1)--(0.5+2,0.95+1);\draw [->] (0.5+2,0.05+2)--(0.5+2,0.5+2)--(0.95+2,0.5+2);
\end{tikzpicture}$$
and if $n=3$, $k=2$ we have
$$
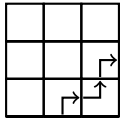
\begin{tikzpicture}[thick,scale=0.5]
\draw (1,0) -- (0,0) -- (0,1)--(1,1)--(1,0)--(0,0);
\draw (1,1) -- (0,1) -- (0,2)--(1,2)--(1,1)--(0,1);
\draw (1,2) -- (0,2) -- (0,3)--(1,3)--(1,2)--(0,2);
\draw (2,0) -- (1,0) -- (1,1)--(2,1)--(2,0)--(1,0);
\draw (2,1) -- (1,1) -- (1,2)--(2,2)--(2,1)--(1,1);
\draw (2,2) -- (1,2) -- (1,3)--(2,3)--(2,2)--(1,2);
\draw (3,0) -- (2,0) -- (2,1)--(3,1)--(3,0)--(2,0);
\draw (3,1) -- (2,1) -- (2,2)--(3,2)--(3,1)--(2,1);
\draw (3,2) -- (2,2) -- (2,3)--(3,3)--(3,2)--(2,2);
\draw [->] (1.5,0.05)--(1.5,0.5)--(1.95,0.5);\draw [->] (1.05+1,0.5)--(1.5+1,0.5)--(1.5+1,0.95);\draw [->] (1.5+1,0.05+1)--(1.5+1,0.5+1)--(1.95+1,0.5+1);
\end{tikzpicture}$$
Therefore,
$\Delta_{\tau}\circ{\mathcal T}(z_k^l)= e^{ i c_k}q^{l-n}\delta_{kl}$ for some constants $c_{k}\in [0,2\pi)$. Moreover, it is easy to see that we can choose the set $\{\tau_{1},\dots,\tau_{n^{2}}\}$ such that $c_{k}=\varphi_{k}$ and hence $\Delta_{\tau}\circ{\mathcal T}(z_k^l)=\chi_\varphi(z_k^l)$. As $\Delta_{\tau}$ is a $*$-homomorphism and hence contractive, we obtain  that $\chi_\varphi$ extends to a $*$-representation of $\mathrm{Pol}(\mathrm{Mat}_n)_q$ and
$$|\chi_\varphi(a)|=\|\Delta_{\tau}\circ{\mathcal T}(a)\|\leq \|{\mathcal T}(a)\|$$
for any $a\in  \mathrm{Pol}(\mathrm{Mat}_n)_q$.
\end{proof}

In what follows we shall also need the so-called  coherent representations which are defined in a way similar to the Fock representation: consider a $\mathrm{Pol}(\mathrm{Mat}_n)_q$-module ${\mathcal H}_\Omega$ determined by
a cyclic vector $\Omega$ such that
$$(z_j^i)^*\Omega=0, \text{ if }(i,j)\ne (1,1)\text{ and }(z_1^1)^*\Omega=e^{-i\varphi}\Omega,$$
for $\varphi\in[0,2\pi)$.
Next we shall prove that the module action gives rise to a bounded $*$-representation  of $\mathrm{Pol}(\mathrm{Mat}_n)_q$ on a Hilbert space $H_\Omega$ which is a completion of ${\mathcal H}_\Omega$ with respect to some inner product.
For $\psi\in[0,2\pi)$ let $\omega_\psi: C(\mathbb T)\to\mathbb C$ be the evaluation map $f\mapsto f(e^{i\psi})$.

\begin{proposition}\label{coherent_more}
Let  $\mathcal T$ be  the Fock representation given in Theorem \ref{Fock_via_suq2n}. Then $$\mathcal T_\psi:=({\rm id}\otimes\ldots\otimes \underset{\stackrel{\uparrow}{n}}{(\omega_\psi\circ \phi)}\otimes{\rm id}\otimes\ldots\otimes{\rm id})\circ \mathcal T:\mathrm{Pol}(\mathrm{Mat}_n)_q\to C^*(S)^{\otimes(n^2-1)}$$
is unitary equivalent to a coherent representation of $\mathrm{Pol}(\mathrm{Mat}_n)_q$.
\end{proposition}
 \begin{proof}
Any admissible route from the position $(n,1)$ to $(1,n)$ is either

$$
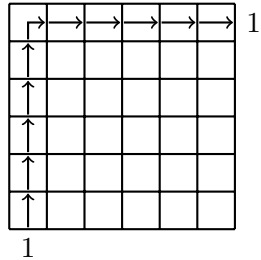
\begin{tikzpicture}[thick,scale=0.5]
\draw[step=1.0,black,thick] (1,1) grid (7,7);
\node at (1.5,0.5) {$1$};
\node at (7.5,6.5) {$1$};
\draw [->] (1/2+1,0.05+1)--(1/2+1,0.95+1);
\draw [->] (1/2+1,0.05+2)--(1/2+1,0.95+2);
\draw [->] (1/2+1,0.05+3)--(1/2+1,0.95+3);
\draw [->] (1/2+1,0.05+4)--(1/2+1,0.95+4);
\draw [->] (1/2+1,0.05+5)--(1/2+1,0.95+5);
\draw [->] (0.5+1,0.05+6)--(0.5+1,0.5+6)--(0.95+1,0.5+6);
\draw [->] (0.05+2,0.5+6)--(0.95+2,0.5+6);
\draw [->] (0.05+3,0.5+6)--(0.95+3,0.5+6);
\draw [->] (0.05+4,0.5+6)--(0.95+4,0.5+6);
\draw [->] (0.05+5,0.5+6)--(0.95+5,0.5+6);
\draw [->] (0.05+6,0.5+6)--(0.95+6,0.5+6);
\end{tikzpicture}
$$
or contains both the right hook and the upper hook arrow boxes in  positions different from $(1,1)$. Hence the corresponding summand in the expression for $\mathcal T(z_1^1)$ is either
$$(-1)^{n-1} d(q)\otimes\ldots\otimes d(q)\otimes\underset{\stackrel{\uparrow}{n}}{C_qS}\otimes I\otimes\ldots\otimes I\otimes \underset{\stackrel{\uparrow}{2n}}{d(q)}\ldots\otimes I\otimes\ldots\otimes I\otimes\underset{\stackrel{\uparrow}{n^2}}{d(q)}$$
or an elementary tensor that contains both the factors $C_qS$ and $S^*C_q$  in positions different from $n=h(1,1)$ and has the identity operator as the $n$-th factor. Hence applying   $({\rm id}\otimes\ldots\otimes \underset{\stackrel{\uparrow}{n}}{(\omega_\psi\circ \phi)}\otimes{\rm id}\otimes\ldots\otimes{\rm id})$ we obtain
$$\mathcal T_\psi(z_1^1)\Omega=(-1)^{n-1}e^{i\psi}\Omega$$
for $\Omega:=e_0\otimes\ldots\otimes e_0\in\ell^2(\mathbb Z_+)^{\otimes (n^2-1)}.$
That $\mathcal T_\psi(z^j_k)^*\Omega=0$ for $(j,k)\ne (1,1)$ follows from the fact that any admissible path from $(n,j)$ to $(k,n)$ contains the right hook arrow box in a position different from $(1,1)$. That $\Omega$ is a cyclic vector can be seen using arguments similar to one in Lemma \ref{density}.
 \end{proof}

 It follows from \cite[Proposition 1.3.3]{jsw} that a coherent representation corresponding to $\varphi\in[0,2\pi)$ is unique, up to unitary equivalence. We shall denote it by $\rho_\varphi^n$.

 The following technical result will be needed in the next section and can be easily derived from our box diagrams.

 \begin{lemma} \label{Fock_coherent} Let $\varphi\in[0,2\pi)$.
 There exist operators $A_{jk}, B\in B(\ell^2(\mathbb Z_+^{(n^2-1)}))$ such that $\pi_{F,n}$ and $\rho_\varphi^n$ are unitary equivalent to the following representations on $\ell^2(\mathbb Z_+^{(n^2-1)})\otimes\ell^2(\mathbb Z_+)$ and
$\ell^2(\mathbb Z_+^{(n^2-1)})$, respectively:
\begin{eqnarray*}
&&\tilde\pi_{F,n}:z_1^1\mapsto B\otimes C_qS+A_{11}\otimes I, \quad z_k^j\mapsto A_{jk}\otimes I,\ (j,k)\ne (1,1)\\
&&\tilde\rho_\varphi^n:z_1^1\mapsto e^{i\varphi}B+A_{11}, \quad z_k^j\mapsto A_{jk}, \ (j,k)\ne (1,1)
\end{eqnarray*}
 \end{lemma}

\begin{proof} Let $\psi\in[0,2\pi)$, $(-1)^{n-1}e^{i\psi}=e^{i\varphi}$ and
let $\mathcal T$ and $\mathcal T_\psi$ be the Fock and the coherent representations from Theorem \ref{Fock_via_suq2n} and Proposition \ref{coherent_more}, respectively.
It follows from the  proof of the previous proposition that
\begin{equation}\label{tz11}
\mathcal T(z_1^1)= R_1^1\otimes \underset{\stackrel{\uparrow}{n}}{C_qS}\otimes F_1^1+\sum_iA_1^1(i)\otimes \underset{\stackrel{\uparrow}{n}}I\otimes B_1^1(i)
\end{equation} and
\begin{equation}\label{tzkj}
\mathcal T(z_k^j)=\sum_i A_k^j(i)\otimes \underset{\stackrel{\uparrow}{n}}I\otimes B_k^j(i),\  (k,j)\ne (1,1),
\end{equation}
for some elementary tensors $R_1^1$, $F_1^1$, $A_k^j(i)$, $B_k^j(i)$.

Applying $({\rm id}\otimes\ldots\otimes \underset{\stackrel{\uparrow}{n}}{(\omega_\psi\circ \phi)}\otimes{\rm id}\otimes\ldots\otimes{\rm id})$ to (\ref{tz11})-(\ref{tzkj}) gives
$$\mathcal T_\psi(z_1^1)= e^{i\psi}R_1^1\otimes F_1^1+\sum_iA_1^1(i)\otimes B_1^1(i)$$
and
$$\mathcal T_\psi(z_k^j)=\sum_i A_k^j(i)\otimes B_k^j(i),\  (k,j)\ne (1,1).$$
Set $A_{jk}=\sum_i A_k^j(i)\otimes B_k^j(i)$ and $B=R_1^1\otimes F_1^1$ and consider the unitary operator $U:\ell^2(\mathbb Z_+)^{\otimes n^2}\to \ell^2(\mathbb Z_+)^{\otimes n^2}$ given by  $$U:f_1\otimes\ldots\otimes f_{n-1}\otimes f_n\otimes f_{n+1}\otimes\ldots\otimes f_{n^2}\mapsto f_1\otimes \ldots\otimes f_{n-1}\otimes f_{n+1}\otimes\ldots\otimes f_{n^2}\otimes f_n.$$
Clearly, $\mathcal T$ is unitary equivalent via the operator $U$ to $\tilde\pi_{F,n}$ giving the statement.
\end{proof}

\medskip

We conclude this section with a couple of lemmas on $*$-representations of ${\rm Pol}({\rm Mat}_n)_q$  induced from the coaction map, defined in the previous section.

 Recall the coaction $\mathcal D_n:{\rm Pol}({\rm Mat}_n)_q\to {\rm Pol}({\rm Mat}_n)_q\otimes \mathbb C[SU_n]_q\otimes \mathbb C[SU_n]_q$ defined in Lemma \ref{coaction} and the $*$-homomorphism  $\Pi_\varphi: {\rm Pol}({\rm Mat}_n)_q\to {\rm Pol}({\rm Mat}_{n-1})_q$, $\varphi\in[0,2\pi)$, given by (\ref{def_Pi_phi}).
Clearly, if $\rho$ is a $*$-representation of $\mathrm{Pol}(\mathrm{Mat}_{n-1})_q$, $\tau$ is a $*$-representation of $\mathrm{Pol}(\mathrm{Mat}_n)_q$, and $\pi_1$, $\pi_2$ are $*$-representations of
$\mathbb C[SU_n]_q$ then $\rho\circ\Pi_\varphi$ and  $(\tau\otimes\pi_1\otimes\pi_2)\circ\mathcal D_n$  are $*$-representations of $\mathrm{Pol}(\mathrm{Mat}_n)_q$.

\begin{lemma}\label{fock_m}
Given $*$-representations $\pi_1$, $\pi_2$ of $\mathbb C[SU_n]_q$, the $*$-representation $(\pi_{F,n}\otimes\pi_1\otimes\pi_2)\circ {\mathcal D}_n$ of $\mathrm{Pol}(\mathrm{Mat}_n)_q$ is a direct sum $\oplus_i\pi_{F,n}$ of copies of the Fock representation. In particular, $\|(\pi_{F,n}\otimes\pi_1\otimes\pi_2)\circ {\mathcal D}_n(a)\|\leq \|\pi_{F,n}(a)\|$ for any $a\in\mathrm{Pol}(\mathrm{Mat}_n)_q$.
\end{lemma}
\begin{proof}
Let $\{f_i\}_{i\in I}$, $\{g_j\}_{j\in J}$ be orthonormal bases for the representation spaces  $H_{\pi_1}$ and $H_{\pi_2}$ respectively and let $v_0$ denote a vacuum vector for  $\pi_{F,n}$. Recall that  $\mathbb C[\mathrm{Mat}_n]_q$ stands for the subalgebra of $\mathrm{Pol}(\mathrm{Mat}_n)_q$ generated by the holomorphic generators $z_l^k$, $k,l=1,\ldots, n$.  We have that each $v_0\otimes f_i\otimes g_j$, $i\in I$, $j\in J$, is a vacuum vector for $\tau:=(\pi_{F,n}\otimes\pi_1\otimes\pi_2)\circ \mathcal D_n$. Moreover, as  $\pi_{F,n}(a)v_0\perp v_0$ for any $a\in \mathbb C[\mathrm{Mat}_n]_q$ with   $\text{deg }a>0$,  one can easily  see  that  $v_0\otimes f_k\otimes g_l$ is orthogonal to   $\text{span}\{\tau(a)(v_0\otimes f_i\otimes g_j):a\in {\mathbb C}[\mathrm{Mat}_n]_q\}$ whenever $(i,j)\ne(k,l)$. As the latter subspace is invariant with respect to $\tau(a)$, $a\in\mathrm{Pol}(\mathrm{Mat}_n)_q$, we have
$$\text{span}\{\tau(a)(v_0\otimes f_i\otimes g_j):a\in {\mathbb C}[\mathrm{Mat}_n]_q\}=\text{span}\{\tau(a)(v_0\otimes f_i\otimes g_j):a\in \mathrm{Pol}(\mathrm{Mat}_n)_q\}$$ and the subspaces
$$\text{span}\{\tau(a)(v_0\otimes f_i\otimes g_j):a\in \mathrm{Pol}(\mathrm{Mat}_n)_q\}\text{ and } \text{span}\{\tau(a)(v_0\otimes f_k\otimes g_l):a\in \mathrm{Pol}(\mathrm{Mat}_n)_q\}$$ are orthogonal whenever $(i,j)\ne (k,l)$.
This implies the statement.
\end{proof}

Let $w=s_{n-1}\ldots s_1\in S_n$ and let $\pi_w$ denote the irreducible representation of $\mathbb C[SU_n]_q$ corresponding to $w$.
\begin{lemma}\label{coherent}
The coherent $*$-representation $\rho_{\varphi}^n$  is unitary equivalent to a sub-representation of $\tau_{\varphi,w,w}:=((\pi_{F,n-1}\circ\Pi_\varphi)\otimes\pi_w\otimes\pi_w)\circ\mathcal D_n$.
\end{lemma}

 \begin{proof}
 It is enough to prove that $\tau_{\varphi,w,w}$ possesses a vector $\Omega$ such that
$\tau_{\varphi,w,w}(z_j^i)^*\Omega=0$ if $(i,j)\ne(1,1)$ and  $\tau_{\varphi,w,w}(z_1^1)^*\Omega=e^{-i\varphi}\Omega$.

 The representation $\tau_{\varphi,w,w}$ acts on $H_{F,n-1}\otimes \ell^2(\mathbb Z_+)^{\otimes(n-1)}\otimes \ell^2(\mathbb Z_+)^{\otimes(n-1)}$.
Let $e_0^{(n-1)}=\underbrace{e_0\otimes\ldots\otimes e_0}_{n-1}$ and $\Omega=v_0\otimes e_0^{(n-1)}\otimes e_0^{(n-1)}$, where $v_0$ is a vacuum vector for $\pi_{F,n-1}$. Then
$$\tau_{\varphi,\omega,\omega}(z_1^1)^*\Omega=\sum_{a,b=1}^n\pi_{F,n-1}(\Pi_\varphi(z_b^a))^*v_0\otimes\pi_w(t_{b,1}^*)e_0^{(n-1)}\otimes\pi_w(t_{a,1}^*)e_0^{(n-1)}.$$
It follows from the definition of the Fock representation and the $*$-homomorphism $\Pi_\varphi$  that  $\pi_{F,n-1}(\Pi_\varphi(z_a^b))^*v_0=0$ whenever $(a,b)\ne (n,n)$ and $\pi_{F,n-1}(\Pi_\varphi(z_n^n))^*v_0=e^{-i\varphi}v_0$.  As
$$\pi_w(t_{n,i}^*)e_0^{(n-1)}=\sum_{k_1,\ldots, k_{n-2}=1}^n\pi_{s_{n-1}}(t_{n,k_1}^*)e_0\otimes \pi_{s_{n-2}}(t_{k_1,k_2}^*)e_0\otimes\ldots\otimes \pi_{s_{1}}(t_{k_{n-2},i}^*)e_0,$$
 $\pi_{s_{k-1}}(t_{k,l}^*)e_0\ne 0$ if and only if $l=k-1$ and $\pi_{s_{k-1}}(t_{k,k-1}^*)e_0=-e_0$, we obtain
$$\pi_w(t_{n,1}^*)e_0^{(n-1)}=\pi_{s_{n-1}}(t_{n,n-1}^*)e_0\otimes\pi_{s_{n-2}}(t_{n-1,n-2}^*)e_0\otimes\ldots\otimes\pi_{s_1}(t_{2,1}^*)e_0=(-1)^{n-1}e_0^{(n-1)}$$
and $\pi_w(t_{n,i}^*)e_0^{(n-1)}=0$ if $i>1$.

Hence
$$\tau_{\varphi,\omega,\omega}(z_1^1)^*\Omega=e^{-i\varphi}v_0\otimes\pi_w(t_{n,1}^*)e_0^{(n-1)}\otimes\pi_w(t_{n,1}^*)e_0^{(n-1)}=e^{-i\varphi}\Omega.$$
Similarly
 $$\tau_{\varphi,\omega,\omega}(z_i^j)^*\Omega=e^{-i\varphi}v_0\otimes\pi_w(t_{n,i}^*)e_0^{(n-1)}\otimes\pi_w(t_{n,j}^*)e_0^{(n-1)}=0\text{ if }(i,j)\ne (1,1).$$
 \end{proof}

\section{Shilov boundary}
The Shilov boundary ideal is a non-commutative analog of the Shilov boundary of a compact Hausdorff space $X$ relative to a uniform subalgebra $\mathcal A$ of  $C(X)$, which is, by definition, the smallest closed subset $K$ of $X$ such that every function in $\mathcal A$ attains its maximum modulus on
 $K$. The notion goes back to the fundamental work by W. Arveson, \cite{Ar}, which gave birth to several directions in mathematics.

A typical ``commutative" example that we should keep in mind is the algebra $\mathcal A$ of holomorphic functions on the open unit ball $\mathbb D_n=\{{\bf z}\in M_n(\mathbb C): {\bf z}{\bf z}^*<1\}$ which are continuous on its closure. The Shilov boundary of $\mathbb D_n$ relative to $\mathcal A$ is known to be the space of unitary $n$ by $n$ matrices.

To introduce the non-commutative version of the Shilov boundary
recall that if $V$ is a subspace of a $C^*$-algebra $\mathcal B$, $M_n(V)$ is the space of $n$ by $n$ matrices in $V$ with norm induced by the $C^*$-norm on $M_n(\mathcal B)$, then any linear map $\phi$ from $V$ to another $C^*$-algebra $\mathcal C$ induces a linear map $\phi^{(n)}:M_n(V)\to M_n(\mathcal C)$ by letting
$$\phi^{(n)}((a_{i,j})_{i,j})=(\phi(a_{i,j}))_{i,j}, \ (a_{i,j})_{i,j}\in M_n(V).$$
The linear map is called  a complete isometry if  $\phi^{(n)}$ is an isometry for any $n$.

Assume $V$ contains the identity of $\mathcal B$ and generates $\mathcal B$ as a $C^*$-algebra.
The following definition was given by Arveson \cite{Ar}.

\begin{definition}
A closed two-sided ideal $J$ in $\mathcal B$ is called a {\it boundary ideal} for $V$ if the canonical quotient map $q:\mathcal B\to \mathcal{B}/J$ is a complete isometry when restricted  to $V$.  A boundary ideal is called a {\it Shilov boundary} or a {\it Shilov  boundary ideal} for $V$ if it contains every other boundary ideal.
\end{definition}

Arveson demonstrated the existence of the Shilov boundary in several situations of particular interests. In \cite{hamana} Hamana proved that the boundary always exists for any such subspace $V$.
%Note that it is also unique (see \cite{Ar}).
Another proof using dilation arguments was given by Dritschel and McCullough in \cite{dritschel_mccullough} (see also \cite{arveson_notes}). As the operator space structure on commutative $C^*$-algebras is minimal  one has  that a closed subset $K$ of a compact Hausdorff space $X$ is the Shilov boundary relative to a uniform subalgebra of $C(X)$ if and only if the ideal $J=\{f\in C(X): f|_K=0\}$ is the Shilov boundary ideal. This shows that the above definition indeed generalises the commutative notion.

If $J$ is the Shilov ideal for a unital subalgebra $V$ of $C^*(V)$, then $C^*(V)/J$ provides a realization of the {\it $C^*$-envelope} $C_e^*(V)$ of $V$, a $C^*$-algebra which is determined by the property: there exists a completely isometric representation $\gamma: V\to  C_e^*(V)$  such that $C^*_e(V)=C^*(\gamma(V))$  and if $\rho$ is any other completely isometric representation, then there exists an onto representation
$\pi:C^*(\rho(V))\to C^*(\gamma(V))$ such that $\pi(\rho(a))=\gamma(a)$ for all $a\in V$.

 \medskip
In this section we shall describe the Shilov boundary ideal for the closed subalgebra of $C_F(\overline{\mathbb D}_n)_q$ generated by the ``holomorphic" generators $\pi_{F,n}(z^i_j)$, $i,j=1,\ldots,n$.

In $\mathrm{Pol}(\mathrm{Mat}_n)_q$ consider the two-sided ideal $J_n$ generated by $$\sum_{j=1}^n q^{2n-\alpha-\beta}z_j^\alpha(z_j^\beta)^*-\delta^{\alpha\beta}, \alpha,\beta=1,\ldots,n,$$ where $\delta^{\alpha\beta}$ is the Kronecker symbol.
The ideal $J_n$ is a $*$-ideal, i.e. $J_n=J_n^*$. The quotient algebra $\mathrm{Pol}(S(\mathbb D_n))_q:=\mathrm{Pol}(\mathrm{Mat}_n)_q/J_n$ is a $U_q{\mathfrak su}_{n,n}$-module $*$-algebra called the polynomial algebra  on the Shilov boundary of a quantum matrix ball.
The canonical homomorphism
$$j_q:\mathrm{Pol}(\mathrm{Mat}_n)_q\to \mathrm{Pol}(S(\mathbb D_n))_q$$
is a $q$-analog of the restriction operator which maps a polynomial on the ball $\mathbb D_n=\{{\bf z}\in \mathrm{Mat}_n:{\bf z}{\bf z}^*< 1\}$ to its restriction to the Shilov boundary $S(\mathbb D_n)=\{{\bf z}\in \mathrm{Mat}_n: {\bf z}{\bf z}^*=1\}$.

The $*$-algebra $\mathrm{Pol}(S(\mathbb D_n))_q$ was introduced by L.Vaksman in \cite{vaksman_shilov1} and shown to be isomorphic to the $*$-algebra $(\mathbb C[GL_n]_q,\ast)\simeq \mathbb C[U_n]_q$ introduced in section \ref{algebra}; the isomorphism $\Psi: \mathrm{Pol}(S(\mathbb D_n))_q\to (\mathbb C[GL_n]_q,\ast)$ is given by $z^i_j+J_n\to z^i_j$, $i,j=1,\ldots,n$ (see \cite[Theorem 2.2, Proposition 6.1]{vaksman_shilov1}).
The author of \cite{vaksman_shilov1} used an algebraic approach to introduce the $q$-analog of the Shilov boundary  and posed a question whether this notion would coincide with the ``analytic" Shilov boundary of Arveson. In this section we shall give an  affirmative answer to that question for general value of $n$.

Let $A(\mathbb D_n)_q$ be the closed (non-involutive) subalgebra of $C_F(\overline{\mathbb D}_n)_q$ generated by $\pi_{F,n}(z_j^i)$, $i,j=1,\ldots,n$. Let $\bar J_n$ be the closure of $J_n$ in $C_F(\overline{\mathbb D}_n)_q$ and write $j_q$ also for the canonical quotient map $C_F(\overline{\mathbb D}_n)_q\to  C_F(\overline{\mathbb D}_n)_q/\bar J_n$.

We are now ready to state the main theorems of the paper.

\begin{theorem}\label{boundary_ideal}
The ideal $\bar J_n$ is a boundary ideal for $A(\mathbb D_n)_q$.
\end{theorem}

\begin{theorem}\label{shilov_boundary}
The ideal $\bar J_n$ is the Shilov boundary ideal for $A(\mathbb D_n)_q$.
\end{theorem}

We begin by  proving  several auxiliary lemmas.
\begin{lemma}\label{shilov}
 Let $\tau$ be a $*$-representation of $\mathrm{Pol}(\mathrm{Mat}_{n-1})_q$ that annihilates the ideal $J_{n-1}$ of  $\mathrm{Pol}({\mathrm Mat}_{n-1})_q$. Then $((\tau\circ\Pi_\varphi)\otimes \pi_1\otimes \pi_2)\circ\mathcal D_{n}$ is a $*$-representation of ${\mathrm{Pol}}(\mathrm{Mat}_{n})_q$ that annihilates the ideal $J_n$ of ${\mathrm{Pol}}(\mathrm{Mat}_{n})_q$ for any $*$-representations $\pi_1$, $\pi_2$ of ${\mathbb C}[SU_n]_q$.
\end{lemma}
\begin{proof}
Write $\Phi_n: \mathrm{Pol}(\mathrm{Mat}_n)_q\to\mathbb C[U_n]_q$ for the $*$-homomorphism given by
$$\Phi_n:z_j^i\mapsto q^{i-n}z_j^i,\ i,j=1,\ldots, n,$$
(see Theorem \ref{vaksman_hom} and the remark after it). It is straightforward to check that
$$\Phi_{n-1}\circ\Pi_\varphi=\Psi_\varphi\circ\Phi_n,\  \varphi\in [0,2\pi],$$ where $\Psi_\varphi: {\mathbb C}[U_n]_q\to {\mathbb C}[U_{n-1}]_q$ is the $*$-homomorphism such that
$$\Psi_\varphi(z^i_j)=\left\{\begin{array}{ll}z^i_j,&i,j<n,\\e^{i\varphi},& i=j=n,\\ 0,&\text{otherwise}.\end{array}\right.$$
Therefore, as $\ker\Phi_n=J_n$, we have $\Pi_\varphi(J_n)\subset J_{n-1}$, giving that $\tau\circ\Pi_\varphi$ annihilates the ideal $J_n$ whenever $\tau$ is a $*$-representation of $\mathrm{Pol}(\mathrm{Mat}_{n-1})_q$ such that $\tau(J_{n-1})=0$.

Next we observe that if $\Upsilon: \mathbb C[U_n]_q\to\mathbb C[SU_n]_q$ is the canonical $*$-homomorphism  given by
$\Upsilon: z^i_j\mapsto t_{i,j}$, then, for any $a\in\mathrm{Pol}(\mathrm{Mat}_n)_q$,
$$
(\Upsilon\otimes\text{id}\otimes\Upsilon)\circ(\Delta\otimes \text{id})\circ\Delta\circ\Phi_n(a)=(\sigma\otimes \text{id})\circ(\text{id}\otimes\sigma)\circ(\text{id}\otimes\text{id}\otimes\theta)\circ(\Phi_n\otimes \text{id}\otimes\text{id})\circ {\mathcal D}_n(a),
$$
where $\sigma$ is the flip map that sends $a\otimes b$ to $b\otimes a$, $\Delta$ is the comultiplication on $\mathbb C[U_n]_q$ and $\theta: \mathbb C[SU_n]_q\to\mathbb C[SU_n]_q$  is the $*$-automorphism defined by (\ref{auto}).
Hence, $(\Phi_n\otimes\text{id}\otimes\text{id})\circ{\mathcal D}_n(J_n)=0$. It is now immediate that if $\rho$ is a $*$-representation of $\mathrm{Pol}(\mathrm{Mat}_n)_q$ that annihilates $J_n$ , then $(\rho\otimes\pi_1\otimes\pi_2)\circ{\mathcal D}_n(J_n)=0$ and, by the first part of the proof,
 $((\tau\circ\Pi_\varphi)\otimes\pi_1\otimes\pi_2)\circ{\mathcal D}_n(J_n)=0$, for any  $\tau$, $\pi_1$, $\pi_2$ satisfying the assumptions of the lemma.

\end{proof}

The dilation technique will play a crucial role in the rest of the paper. Recall that if $V$ is  a vector space and $\phi_i:V\to B(H_i)$, $i=1,2$, are linear maps then  $\phi_2$ is said to be a {\it dilation} of $\phi_1$ on $V$ and write $\phi_1\prec \phi_2$ if $H_1\subset H_2$  and $\phi_1(a)=P_{H_1}\phi_2(a)|_{H_1}$ for $a\in V$. Clearly, the relation $\prec$ on $V$  is transitive.

The next statement shows that the Fock representation is a dilation of a "sum" of coherent representations when considered as maps on  ${\mathbb C}[\mathrm{Mat}_n]_q$.

\begin{lemma}\label{dilation_fock}
There is a $*$-representation $\Psi$ of $\mathrm{Pol}(\mathrm{Mat}_n)_q$   such that $\Psi$ is a direct integral of  coherent representations $\rho_\varphi^n$  and  $\pi_{F,n}\prec \Psi$ on ${\mathbb C}[\mathrm{Mat}_n]_q$.
\end{lemma}

\begin{proof} Let $S$ and $C_q$ be the shift and the diagonal operators on $\ell^2(\mathbb Z_+)$ given by (\ref{SC}).
By Lemma \ref{Fock_coherent},
$$\pi_{F,n}(z_1^1)=B\otimes C_qS+A_{11}\otimes 1,\  \pi_{F,n}(z_j^i)=A_{ij}\otimes 1,\  (i,j)\ne(1,1),$$ for some $A_{ij},B\in B(\ell^2(\mathbb Z_+)^{\otimes (n^2-1)})$,
while the images under  the coherent representations are:
$$\rho_\varphi^n(z_1^1)=(-1)^{n-1}e^{i\varphi}B+A_{11}, \ \rho_\varphi^n(z_j^i)=A_{ij}\ (i,j)\ne(1,1).$$ As $C_qS$ is a contraction, by Sz.-Nagy's theorem (see e.g. \cite{paulsen}),  there exists a unitary operator $U$ on a larger space  containing $\ell^2(\mathbb Z_+)$ such that $(C_qS)^k=P_{\ell^2(\mathbb Z_+)}U^k|_{\ell^2(\mathbb Z_+)}$ for all $k\in\mathbb N$.  Let  $$\Psi(z_1^1)=B\otimes U+A_{11}\otimes 1,\ \Psi(z_j^i)=A_{ij}\otimes 1 \text{ if }(i,j)\ne(1,1).$$  Clearly, $\pi_{F,n}\prec \Psi$ on ${\mathbb C}[{\mathrm{Mat}}_n]_q$ and  $\Psi$ is a $*$-representation of $\mathrm{Pol}(\mathrm{Mat}_n)_q$ which is a direct integral of those $\rho_\varphi^n$ such that  $(-1)^{n-1}e^{i\varphi}$ is in the spectrum of $U$.

\end{proof}

\begin{lemma}\label{fock_2}
\begin{enumerate}
\item Any $*$-representation of $\mathrm{Pol}(\mathrm{Mat}_n)_q$ that annihilates the ideal $J_n$ is a direct integral of $*$-representations given by $z_l^k\mapsto e^{i\varphi_k}q^{k-n}\rho(t_{k,l})$, $k,l=1,\ldots,n$,  where  $\rho$ is a $*$-representation of $\mathbb C[SU_n]_q$ and $\varphi_k\in[0, 2\pi)$.
\item If $\pi$ is a $*$-representation of $\mathrm{Pol}(\mathrm{Mat}_n)_q$ such that $\pi(J_n)=0$  then $\|\pi(a)\|\leq\|\pi_{F,n}(a)\|$, $a\in \mathrm{Pol}(\mathrm{Mat}_n)_q$, and
\begin{equation}\label{jq}
\|j_q^{(k)}((\pi_{F,n}(a_{i,j}))_{i,j})\|=\sup\{\|(\pi(a_{i,j}))_{i,j}\|:\pi\in\mathrm{Rep}(\mathrm{Pol}(\mathrm{Mat}_n)_q),\pi(J_n)=0\}
\end{equation}
 for any $k\in\mathbb N$ and  $(a_{i,j})_{i,j}\in M_k(\mathrm{Pol}(\mathrm{Mat}_n)_q)$.
\end{enumerate}
\end{lemma}

\begin{proof} 1. Let $\pi$ be a $*$-representation of $\mathrm{Pol}(\mathrm{Mat}_n)_q$ that annihilates $J_n$. By \cite[Proposition 6.1, Theorem 2.2]{vaksman_shilov1}, the family $\{\pi(z_j^i):i,j=1,\ldots, n\}$ determines a $*$-representation of $(\mathbb C[GL_n]_q,*)$. Since $\det_q\mathbf{z}$ is central in  $(\mathbb C[GL_n]_q,*)$ and  satisfies $\det_q\mathbf{z}(\det_q\mathbf{z})^*=q^{-n(n-1)}$ (see (\ref{detqz})),
by the spectral theorem
$$\pi(\det\nolimits_q\mathbf{z})=\int_{[0,2\pi]}^{\oplus}e^{i\varphi}q^{-n(n-1)/2}I_\varphi d\mu(\varphi)$$  and
$$\pi=\int_{[0,2\pi]}^{\oplus}\pi_\varphi d\mu(\varphi),$$
where $\{\pi_\varphi:\varphi\in[0,2\pi)\}$ is a field of $*$-representations of $\mathrm{Pol}(\mathrm{Mat}_n)_q$ such that $\pi_\varphi(\det_q\mathbf{z})=e^{i\varphi}q^{-n(n-1)/2}I_\varphi$ with $I_\varphi$ being the identity operator on the representation space of $\pi_\varphi$.

Fix now $\varphi\in[0,2\pi)$ and let $\iota:(\mathbb C[GL_n]_q,\ast )\to\mathbb C[U_n]_q$ be the $*$-isomorphism given by $\iota(z_l^k)=e^{i\varphi_k}q^{k-n}z_l^k$, $k,l=1,\ldots, n$, where $\varphi_1+\ldots+\varphi_n=\varphi$.
Then $\pi_\varphi\circ\iota^{-1}$ is a $*$-representation of $\mathbb C[U_n]_q$ such that $\pi_\varphi\circ\iota^{-1}(\det_q\mathbf{z})=I$ and hence
$\rho_\varphi:=\pi_\varphi\circ\iota^{-1}\circ j^{-1}$ is a $*$-representation of $\mathbb C[SU_n]_q\simeq^j \mathbb C[U_n]_q/\langle\det_q\mathbf{z}-1\rangle$, where $j(z_l^k)=t_{k,l}$, $k,l=1,\ldots,n$. We obtain
$$\pi_\varphi(z_l^k)=\rho_\varphi\circ j\circ \iota(z_l^k)=e^{i\varphi_k}q^{k-n}\rho_\varphi(t_{k,l}), k,l=1,\ldots,n.$$

2. Let $\chi_{\bf\varphi}$, $\varphi=(\varphi_1,\ldots,\varphi_n)$, be the one dimensional representation of $\mathrm{Pol}(\mathrm{Mat}_n)_q$ defined in Lemma \ref{character}. By the lemma the mapping $\psi:\pi_{F,n}(a)\mapsto \chi_{\bf \varphi}(a)$, $a\in\mathrm{Pol}(\mathrm{Mat}_n)_q$ extends to a $*$-homomorphism from $C_F(\overline{\mathbb{D}}_n)_q$ to $\mathbb C$.
Given  representations $\pi_1$, $\pi_2$ of $\mathbb{C}[SU_n]_q$ and $a\in \mathrm{Pol}(\mathrm{Mat}_n)_q$, by Lemma \ref{fock_m}, we obtain
\begin{eqnarray}\label{ineq}
\|(\chi_{\bf \varphi}\otimes\pi_1\otimes\pi_2)\circ {\mathcal D}_n(a)\|&=&\|(\psi\otimes\mathrm{id}\otimes\mathrm{id})((\pi_{F,n}\otimes\pi_1\otimes\pi_2)\circ {\mathcal D}_n(a))\|\\
&\leq& \|(\pi_{F,n}\otimes\pi_1\otimes\pi_2)\circ {\mathcal D}_n(a)\|\leq
\|\pi_{F,n}(a)\|.\nonumber
\end{eqnarray}

If $\pi_2$ is the one-dimensional representation given by $\pi_2(t_{k,l})=\delta_{kl}$, $k,l=1,\ldots,n$, then
$$(\chi_{\bf\varphi}\otimes\pi_1\otimes\pi_2)\circ \mathcal D_n(z_l^k)=\sum_{a,b=1}^n\chi_{\bf\varphi}(z_b^a)\otimes\pi_1(t_{b,l})\otimes\pi_2(t_{a,k})=e^{i\varphi_k}q^{k-n}\pi_1(t_{k,l}).$$
Hence, applying the first statement of the lemma and (\ref{ineq}) we get
$$\|\pi(a)\|\leq\|\pi_{F,n}(a)\|,\ a\in \mathrm{Pol}(\mathrm{Mat}_n)_q$$
for any $*$-representation $\pi$ of $\mathrm{Pol}(\mathrm{Mat}_n)_q$ such that $\pi(J_n)=0$.

 The equality (\ref{jq}) holds by  the fact that $\|j_q^{(k)}((\pi_{F,n}(a_{i,j}))_{i,j})\|$ is the supremum over  $\|(\pi(a_{i,j})_{i,j}\|$, where $\pi$ runs over all $*$-representations of $\mathrm{Pol}(\mathrm{Mat}_n)_q$ that annihilate the ideal $J_n$ and such that $\|\pi(x)\|\leq\|\pi_{F,n}(x)\|$, $x\in \mathrm{Pol}(\mathrm{Mat}_n)_q$.
 \end{proof}

\medskip
We are now in a position to prove the main theorems.

\medskip

{\bf Proof of Theorem \ref{boundary_ideal}.}

Since $j_q$ is a $*$-homomorphism between $C^*$-algebras, it is a complete contraction. To see that $j_q$ is a complete isometry when restricted to $A(\mathbb D_n)_q$
it is enough to prove that
$$||(\pi_{F,n}(a_{i,j}))_{i,j}||\leq ||j_q^{(k)}((\pi_{F,n}(a_{i,j}))_{i,j})||$$
for any $(a_{i,j})_{i,j}\in  M_k(A(\mathbb D_n)_q)$.

The strategy is to find a $*$-representation $\Pi:\mathrm{Pol}(\mathrm{Mat}_n)_q\to B(K)$, $K\supset H_{F,n}$ such that $\Pi(J_n)=0$ and
\begin{equation}\label{dilation}
\pi_{F,n}(a)=P_{H_{F,n}}\Pi(a)|_{H_{F,n}}, a\in \mathbb C[{\rm Mat}_n]_q,
\end{equation}
i.e. $\pi_{F,n}\prec \Pi$ when considered as maps on $\mathbb C[{\rm Mat}_n]_q$. As in this case
$$\|(\pi_{F,n}(a_{i,j}))_{i,j}\|\leq \|(\Pi(a_{i,j}))_{i,j}\|, (a_{i,j})_{i,j}\in M_k(\mathbb C[{\rm Mat}_n]_q), k\in\mathbb N,$$
Lemma  \ref{fock_2} would give immediately the statement.

 We proceed by induction. If $n=1$ this was proved in \cite{vaksman-boundary}. For the reader's convenience we reproduce the arguments. We have that $\mathrm{Pol}(\mathbb C)_q$ is generated by a single element $z$ subject to the relation $z^*z=q^2zz^*+(1-q^2)$ and $\pi_{F,1}(z)=C_qS$, a contraction (see e.g. the proof of Lemma \ref{character}). By Sz.-Nagy's theorem there exist a Hilbert space $K\supset H_{F,1}$ and a unitary operator $U\in B(K)$ such that $\pi_{F,1}(p(z))=P_{H_{F,1}}p(U)|_{H_{F,1}}$ for all holomorphic polynomials $p$. As the map $z\mapsto U$ extends to a  $*$-representation of ${\rm Pol}(\mathbb C)_q$ that annihilates the ideal $J_1=\langle zz^*-1\rangle$, we obtain  a necessary dilation.

  Assume (\ref{dilation})  holds for some $n\geq 1$.
By Lemma \ref{dilation_fock} there exists a $*$-representation $\Psi$ of $\mathrm{Pol}(\mathrm{Mat}_{n+1})_q$, which is a direct integral of coherent representations $\rho_{\varphi}^{n+1}$ and such that  $\pi_{F,n+1}\prec \Psi$ on  $\mathbb C[{\rm Mat}_{n+1}]_q$.

By the remark after Lemma \ref{coherent}, $\rho_\varphi^{n+1}$ is a $*$-subrepresentation of $\tau_{\varphi,w,w}$ and hence $\rho_\varphi^{n+1}(a)=P_L\tau_{\varphi,w,w}(a)|_L$, $a\in \mathrm{Pol}(\mathrm{Mat}_{n+1})_q$,  for a subspace $L$.

As $\tau_{\varphi,w,w}=((\pi_{F,n}\circ\Pi_\varphi)\otimes\pi_w\otimes\pi_w)\circ {\mathcal D}_{n+1}$ and by induction $\pi_{F,n}\prec\Pi$ on  $\mathbb C[{\rm Mat}_n]_q$ with $\Pi(J_n)=0$
we obtain  $$\tau_{\varphi,w,w}\prec((\Pi\circ\Pi_\varphi) \otimes\pi_w\otimes\pi_w)\circ {\mathcal D}_{n+1} \text { on } \mathbb C[{\rm Mat}_{n+1}]_q.$$ Finally,  by Lemma \ref{shilov}
$((\Pi\circ\Pi_\varphi) \otimes\pi_w\otimes\pi_w)\circ{\mathcal D}_{n+1}$ annihilates the ideal $J_{n+1}$  of $\mathrm{Pol}(\mathrm{Mat}_{n+1})_q$.

Combining all these steps and using transitivity of the relation $\prec$  we obtain the desired statement.
\qed

\medskip

{\bf Proof of Theorem \ref{shilov_boundary}.}

Assume that $I$ is a boundary ideal for $ A(\mathbb D_n)_q$ with $I\supset \bar J_n$ and  identify $\mathrm{Pol}({\rm Mat}_n)_q$ with its image under the Fock representation.
Let $a\in \mathbb C[{\rm Mat}_n]_q\subset{\rm Pol}({\rm Mat}_n)_q$ be a polynomial in ``holomorphic" generators $z_j^i$, $1\leq i,j\leq n$.

By assumption,
\begin{equation}\label{idealij}
\|a+I\|=\|a\|=\|a+\bar J_n\|.
\end{equation}
Let $x\in {\rm Pol}({\rm Mat}_n)_q$.
%g((z_j^i),((z_j^i)^*))$ be a polynomial in $z_j^i$ and $(z_j^i)^*$, $1\leq i,j\leq n$.
  As,  by  (\ref{star_pol}) and  (\ref{detqz}), $$(z_j^i)^*+J_n=(-q)^{i+j-2n}(\det\nolimits_q{\bf z})^{-1}\det\nolimits_q{\bf z}_j^i+J_n,$$ $\det_q{\bf z}+J_n$ is a central element in $\mathrm{Pol}(\mathrm{Mat}_n)_q/J_n$ and $$(\det\nolimits_q{\bf z})^*\det\nolimits_q{\bf z}+J_n=q^{-n(n-1)}+J_n,$$ there exist $k\in\mathbb Z_+$ and $a\in\mathbb C[{\rm Mat}_n]_q$ such that $x+\bar J_n=(\det_q{\bf z})^{-k}a+\bar J_n$. Hence, by (\ref{idealij}) and the fact that $I\supset\bar J_n$, we obtain
\begin{eqnarray*}
\|x+\bar J_n\|^2&=&\|(x^*+\bar J_n)(x+\bar J_n)\|\\
&=&\|a^*((\det\nolimits_q{\bf z})^{-k})^*(\det\nolimits_{q}{\bf z})^{-k}a+\bar J_n\|\\
&=&\|a^*q^{kn(n-1)}a+\bar J_n\|=\|q^{kn(n-1)/2}a+\bar J_n\|^2\\
&=&\|q^{kn(n-1)/2}a+I\|^2=\|a^*q^{kn(n-1)}a+I\|\\
&=&\|a^*((\det\nolimits_q{\bf z})^{-k})^*(\det\nolimits_q{\bf z})^{-k}a+I\|\\
&=&\|x+I\|^2.
\end{eqnarray*}
This implies that $C_F(\overline{\mathbb D}_n)_q/\bar J_n=C_F(\overline{\mathbb D}_n)_q/I$ and hence $\bar J_n=I$. \qed

\medskip

Let $C(U_n)_q$ be the $C^*$-enveloping algebra of ${\mathbb C}[U_n]_q$. As the matrix $(z^i_j)_{i,j}$ formed by the generators of ${\mathbb C}[U_n]_q$ is unitary one has that the norm of each generator is not larger than 1  in each $*$-representation by bounded operators on a Hilbert space and hence the $C^*$-enveloping algebra is well-defined.
\begin{cor}
The $C^*$-envelope $C_e^*(A(\mathbb D_n)_q)$ is isomorphic to $C(U_n)_q$.
\end{cor}

\begin{proof} Let $\psi:{\rm Pol}({\rm Mat}_n)_q\to ({\mathbb C}[GL_n]_q,\ast)$ be the surjective $*$-homomorphism from Theorem \ref{vaksman_hom} and
$\iota:({\mathbb C}[GL_n]_q,\ast )\to{\mathbb C}[U_n]_q$ the $*$-isomorphism given by $\iota(z_l^k)=q^{k-n}z_l^k$, $k,l=1,\ldots, n$.
As $\ker\psi=J_n$, any $*$-representation of ${\rm Pol}({\rm Mat}_n)_q$ such that $\pi(J_n)=0$ is given by $\pi(a)=\rho(\iota\circ\psi(a))$, $a\in {\rm Pol}({\rm Mat}_n)_q,$ for some $*$-representation $\rho$ of ${\mathbb C}[U_n]_q$. Moreover, the correspondence $\pi\leftrightarrow\rho$ is one-to-one.
Consider a $*$-representation $\rho$ of $\mathbb C[U_n]_q$ such that $\overline{\rho({\mathbb C}[U_n]_q)}\simeq C(U_n)_q$. In what follows we identify the latter two algebras. Let
$\Psi:\pi_{F,n}({\rm Pol}({\rm Mat}_n)_q)\to C(U_n)_q$, $\pi_{F,n}(a)\mapsto\rho(\iota\circ\psi(a))$. By Lemma \ref{fock_2},
$$\|\rho(\iota\circ\psi(a))\|\leq\|\pi_{F,n}(a)\|, a\in {\rm Pol}({\rm Mat}_n)_q$$
and hence $\Psi$ extends to a surjective $*$-homomorphism (denoted by the same letter) from $C_F(\overline{\mathbb D}_n)_q$ to $C(U_n)_q$. As any representation of $C(U_n)_q$ gives rise to a representation of $C_F(\overline{\mathbb D}_n)_q$ that annihilates the ideal $J_n$, we have $\ker\Psi=\bar J_n$ and hence
$$C_F(\overline{\mathbb D}_n)_q/\bar J_n\simeq C(U_n)_q.$$
As $\bar J_n$ is the Shilov boundary ideal of $A(\mathbb D_n)_q$, $C_F(\overline{\mathbb D}_n)_q/\bar J_n\simeq  C_e^*(A(\mathbb D_n)_q)$ giving the statement.
\end{proof}

\vspace{0.5cm}

\noindent {\bf Acknowledgements. } We would like to thank Daniil Proskurin for valuable conversations during the preparation of this paper. We are grateful to the referee for a number of suggestions that improved the exposition.

\end{document}